\documentclass[preprint,11pt,numbers]{elsarticle-arxiv}

\usepackage{amssymb}
\usepackage{amsmath}
\usepackage{amsthm}
\usepackage{mathtools}
\usepackage{array}
\usepackage{tikz}
\usetikzlibrary{calc,positioning}
\usepackage{url}

\setcitestyle{square,comma}

\newtheorem{theorem}{Theorem}[section]
\newtheorem{corollary}[theorem]{Corollary}
\newtheorem{conjecture}{Conjecture}[section]
\newtheorem{definition}{Definition}[section]
\newtheorem{lemma}[theorem]{Lemma}
\newtheorem{problem}[theorem]{Problem}

\newcommand{\calF}{\mathcal{F}}
\newcommand{\codeg}{\operatorname{codeg}}
\newcommand{\cdeg}{\operatorname{cdeg}}
\numberwithin{equation}{section}

\journal{The Electronic Journal of Combinatorics}
\begin{document}
\bibliographystyle{alpha}

\begin{frontmatter}

\title{\textbf{\Large $(2,\calF)$-Avoiding Coloring and B-Coloring under Bipartite Exclusions}}

\author[1]{Zhijun Lu}
\author[1]{Qirui Ying}
\author[1]{Huimin Song\corref{cor1}}
\cortext[cor1]{Corresponding author. \textit{E-mail address:} \url{hmsong@sdu.edu.cn} (H.Song).}

\affiliation[1]{organization={School of Mathematics and Statistics},
            addressline={Shandong University},
            city={Weihai},
            postcode={264209},
            state={Shandong},
            country={China}}

\begin{abstract}
\ \ \ \ Let $\calF$ be a nonempty family of connected bipartite graphs, each with at least two edges. For a graph $G$, a proper vertex coloring of $G$ is $(2,\calF)$-avoiding if no member of
$\calF$ occurs bichromatically, and $\chi_{2,\calF}(G)$ denotes the minimum
number of colors in such a coloring. A B-coloring of $G$ is a proper
edge-coloring in which every $4$-cycle is rainbow, and $q_B(G)$ denotes the
minimum number of colors in a B-coloring of $G$. In this paper, we introduce
a structural parameter of bipartite graphs and estimate
$\chi_{2,\calF}(G)$ and $q_B(G)$ in terms of this parameter. For a fixed connected bipartite graph $F$ with at least one edge and bipartition classes $X_F$ and $Y_F$, define
$
k(F)=\min\bigl\{|I|: I\subseteq X_F\text{ or }I\subseteq Y_F,
\ F-I\text{ is a forest}\bigr\}.
$

Let $m\ge2$ be the minimum number of edges in a member of $\calF$. We prove
that if $k(F)\le m-2$, then every $F$-free graph $G$ of sufficiently large
maximum degree $\Delta$ satisfies $\chi_{2,\calF}(G)=
O((\frac{\Delta^m}{\log\Delta})^{\frac{1}{m-1}})$, which gives a positive
answer to Chuet's Problem~A and C in a sharp sense, thereby extending the results
of Chuet [arXiv:2603.23379] from frugal colorings to
$(2,\calF)$-avoiding colorings.

For B-colorings, put $k=k(F)$, $h=|V(F)|$, and
$s=\min\{|X_F|,|Y_F|\}$. We prove that every $F$-free graph $G$ of
sufficiently large maximum degree $\Delta$ satisfies
\[
q_B(G)\le
\begin{cases}
\Delta+\Delta^{1-\eta}+1, & \text{if }s\le2,\\[4pt]
(4h-2)(\Delta-1)+1, & \text{if }s\ge3\text{ and }k\le1,\\[4pt]
C\frac{\Delta^{2-\frac{1}{k}}}{\log\Delta},
& \text{if }k\ge2,
\end{cases}
\]
where $\eta>0$ and $C>0$ depend only on $F$. For $k\le1$, the linear order is best possible. For every integer $k\ge2$ and every $\varepsilon>0$, we show that
there exists a fixed connected bipartite graph $F$ with $k(F)=k$
such that, for every sufficiently large integer $\Delta$, there is an
$F$-free graph $G$ of maximum degree $\Delta$ satisfying
$q_B(G)=\Omega\!\left(
\frac{\Delta^{2-\frac{1}{k}-\varepsilon}}{\log\Delta}\right)$.

To prove these results, we develop a common reduction of the coloring
problems to $P$-perfect matching problems in auxiliary hypergraphs and apply
the forbidden-submatching theorem of Delcourt and Postle.
\end{abstract}


\end{frontmatter}

\section{Introduction}\label{s:intro}

\subsection{Preliminaries}\label{ss:prelim}

All graphs considered are simple, finite, and undirected. For a graph $G$,
let $V(G)$ and $E(G)$ denote its vertex set (always nonempty) and edge set,
respectively. The neighborhood of a vertex $v$ in $G$ is
$N_G(v)=\{u\in V(G):uv\in E(G)\}$, and its elements are called the
neighbors of $v$. The degree of $v$, denoted by $d_G(v)$, is
$|N_G(v)|$. The maximum degree $\Delta(G)$ and the minimum degree
$\delta(G)$ are the maximum and minimum degree among all vertices of $G$,
respectively.

A graph $H$ is a subgraph of a graph $G$, denoted by $H\subseteq G$, if
$V(H)\subseteq V(G)$ and $E(H)\subseteq E(G)$. A subgraph $H$ of $G$ is
called an induced subgraph of $G$ if
$E(H)=\{uv\in E(G):u,v\in V(H)\}$. For a set $S\subseteq V(G)$, the
subgraph of $G$ induced by $S$ is denoted by $G[S]$. The \emph{square}
$G^2$ of $G$ is the graph with vertex set $V(G)$ in which two distinct
vertices are adjacent if they are adjacent in $G$ or have a common neighbor
in $G$. We say that $G$ is \emph{$H$-free} if it contains no copy of $H$. The complete bipartite graph
with bipartition classes of sizes $s$ and $t$ is denoted by $K_{s,t}$.
For $r\ge1$, let $P_r$ denote the path on $r$ vertices, and for
$r\ge3$, let $C_r$ denote the cycle on $r$ vertices.
The girth $g(G)$ is the length of a shortest cycle in $G$, and
$g(G)=\infty$ if $G$ is acyclic. When no ambiguity arises, we abbreviate
$d_G(v)$, $N_G(v)$, and $g(G)$ as $d(v)$, $N(v)$, and $g$, respectively.

A graph is \emph{$d$-degenerate} if every nonempty subgraph has a vertex
of degree at most $d$.

For a positive integer $k$, let $[k]=\{1,2,\ldots,k\}$. A set of
cardinality exactly $k$ is called a $k$-set. All logarithms are natural.
Constants implicit in asymptotic notation may depend on the fixed parameters.

A proper $k$-coloring of $G$ is a mapping $\varphi:V(G)\to[k]$ such
that $\varphi(u)\ne\varphi(v)$ for every edge $uv\in E(G)$. The minimum such $k$ is the
chromatic number $\chi(G)$. A subgraph is \emph{bichromatic} if its vertices
receive at most two colors. Let $\calF$ be a nonempty family of connected
bipartite graphs, each with at least two edges. A proper coloring is \emph{$(2,\calF)$-avoiding} if no
member of $\calF$ occurs bichromatically, and $\chi_{2,\calF}(G)$ denotes
the minimum number of colors in such a coloring. When $\calF=\{H\}$, we
write $\chi_{2,H}(G)$. We also write
$\chi_{2,\calF}(\Delta)=\max\{\chi_{2,\calF}(G):\Delta(G)\le\Delta\}$, and $\chi_{2,\calF}(\Delta,F) = \max\{\chi_{2,\calF}(G):\Delta(G)\le\Delta, G \text{ is $F$-free}\}$ 
for each positive integer $\Delta$ and each fixed graph $F$ with at least
one edge. We use $\chi_{2,H}(\Delta)$ and
$\chi_{2,H}(\Delta,F)$ when $\calF=\{H\}$.

A proper coloring is \emph{$\beta$-frugal} if no color appears more than
$\beta$ times in the neighborhood of any vertex, and $\chi_\beta(G)$ denotes
the minimum number of colors in such a coloring. Thus $\beta$-frugal
colorings are precisely $(2,K_{1,\beta+1})$-avoiding colorings. Accordingly,
we write $\chi_\beta(\Delta,F)=\chi_{2,K_{1,\beta+1}}(\Delta,F)$.
Taking $\calF=\{P_4\}$ gives star coloring, while taking
$\calF=\{C_{2t}:t\ge2\}$ gives acyclic coloring.
For $\calF=\{P_3\}$, a $(2,\calF)$-avoiding coloring is precisely a
$1$-frugal coloring, and equivalently a proper coloring of $G^2$.
Thus $\chi_{2,P_3}(G)=\chi_1(G)=\chi(G^2)$.

We also consider edge-colorings. A proper edge-coloring of $G$ is a mapping
$\varphi:E(G)\to[k]$ such that adjacent edges receive distinct colors.
A subgraph of an
edge-colored graph is \emph{rainbow} if all its edges receive distinct
colors. A \emph{B-coloring} of $G$ is a proper edge-coloring in which every
copy of $C_4$ is rainbow. The minimum number of colors in a B-coloring of
$G$ is denoted by $q_B(G)$. For a positive integer $\Delta$ and a fixed
graph $F$ with at least one edge, put
$q_B(\Delta,F)=\max\{q_B(G):\Delta(G)\le\Delta,\ G\text{ is $F$-free}\}$.
A \emph{strong edge-coloring} is a proper
edge-coloring in which each color class is an induced matching, and the
minimum number of colors in such a coloring is the strong chromatic index
$\chi'_s(G)$.

We next introduce the hypergraph terminology used in the proofs. The
formulation follows Delcourt and Postle~\cite{DelcourtPostle2024}. A
hypergraph $\mathcal A$ consists of a vertex set $V(\mathcal A)$ and a set
$E(\mathcal A)$ of subsets of $V(\mathcal A)$. For
$x\in V(\mathcal A)$, let
$d_{\mathcal A}(x)=|\{e\in E(\mathcal A):x\in e\}|$.
For $x,y\in V(\mathcal A)$ with $x\ne y$, their pair-codegree is
$\codeg_{\mathcal A}(x,y)
=|\{e\in E(\mathcal A):\{x,y\}\subseteq e\}|$, and we put
$\codeg(\mathcal A)
=\max\limits_{x\ne y}
\codeg_{\mathcal A}(x,y)$. We call $\mathcal A$
$r$-bounded if every edge has size at most $r$.

A matching in $\mathcal A$ is a set of pairwise disjoint edges. We call
$\mathcal A$ \emph{bipartite with parts $P,Q$}
if $V(\mathcal A)=P\cup Q$, $P\cap Q=\emptyset$, and $|e\cap P|=1$ for every
$e\in E(\mathcal A)$. A matching in $\mathcal A$ is \emph{$P$-perfect}
if it covers every vertex of $P$.

\begin{definition}\label{def:configuration}
Let $\mathcal A$ be a hypergraph. A hypergraph $\mathcal C$ is a
\emph{configuration hypergraph for $\mathcal A$} if
$V(\mathcal C)=E(\mathcal A)$ and every edge of $\mathcal C$ is a matching
of $\mathcal A$ of size at least two. The edges of $\mathcal C$ are called
\emph{configurations}. A matching $M$ of $\mathcal A$ is
\emph{$\mathcal C$-avoiding} if $F\nsubseteq M$ for every
$F\in E(\mathcal C)$.
\end{definition}

For $i\ge2$, let
$E_i(\mathcal C)=\{F\in E(\mathcal C):|F|=i\}$. For
$e\in V(\mathcal C)=E(\mathcal A)$, let
$d_{\mathcal C,i}(e)=|\{F\in E_i(\mathcal C):e\in F\}|$, and define
$\Delta_i(\mathcal C)=\max\limits_{e\in V(\mathcal C)}d_{\mathcal C,i}(e)$.
For integers $2\le\ell<k$, define
\[
\Delta_{k,\ell}(\mathcal C)
=
\max\limits_{S\in\binom{V(\mathcal C)}{\ell}}
|\{F\in E_k(\mathcal C):S\subseteq F\}|.
\]

For $x\in V(\mathcal A)$ and $e\in E(\mathcal A)=V(\mathcal C)$ with
$x\notin e$, define
\[
\codeg_2(\mathcal A,\mathcal C;x,e)
=
|\{F\in E_2(\mathcal C):e\in F,\;
\exists f\in F\setminus\{e\}\text{ with }x\in f\}|,
\]
and let
$\codeg_2(\mathcal A,\mathcal C)
=\max\limits_{x\notin e}
\codeg_2(\mathcal A,\mathcal C;x,e)$.
For $e,f\in V(\mathcal C)$ with $e\ne f$, define their common $2$-degree by
$\cdeg_2(\mathcal C;e,f) = |\{w\in V(\mathcal C): \{e,w\},\{f,w\}\in E_2(\mathcal C)\}|$,
and let
$\cdeg_2(\mathcal C)
=\max\limits_{e\ne f}
\cdeg_2(\mathcal C;e,f)$.

We refer to Bondy and
Murty~\cite{BondyMurty2008} for graph-theoretic terminology not explicitly
defined in this paper.

\subsection{Background and Main Results}\label{ss:background}

In 1996, Johansson~\cite{Johansson1996} proved that every
triangle-free graph of maximum degree $\Delta$ has
chromatic number $O(\frac{\Delta}{\log\Delta})$,
while the random constructions of Bollob\'as~\cite{Bollobas1981}
show that the order $\frac{\Delta}{\log\Delta}$ remains best possible
under any fixed prescribed girth.
In a related direction, Alon, Krivelevich, and Sudakov~\cite{AlonKrivelevichSudakov1999} placed
Johansson's result in a broader local-sparsity framework and posed the following conjecture.

\begin{conjecture}[{\cite{AlonKrivelevichSudakov1999}}]\label{conj:AKS}
For every fixed graph $F$, there exist a constant $C>0$ and $\Delta_0$
such that every $F$-free graph $G$ of maximum degree
$\Delta\ge\Delta_0$ satisfies
$\chi(G)\le C\frac{\Delta}{\log\Delta}$.
\end{conjecture}

For an arbitrary fixed graph $F$, the best known general upper bound
for the chromatic number of an $F$-free graph of maximum degree
$\Delta$ is
$O(\frac{\Delta\log\log\Delta}{\log\Delta})$~\cite{Johansson1996,Molloy2019},
and the strongest known bounds for several important classes of forbidden
graphs are summarized in \textbf{Table~\ref{tab:sparse-coloring}}.

\begin{table}[htbp]
\centering
\caption{Representative bounds toward \textbf{Conjecture~\ref{conj:AKS}}. The excluded
graph $F$ is fixed and $\Delta\to\infty$.}
\label{tab:sparse-coloring}
\renewcommand{\arraystretch}{1.18}
\small
\begin{tabular}{|>{\centering\arraybackslash}p{0.26\textwidth}|
>{\centering\arraybackslash}p{0.34\textwidth}|
>{\centering\arraybackslash}p{0.25\textwidth}|}
\hline
Excluded graph $F$ & Best upper bound for $\chi(G)$ & Reference\\
\hline
$F$ a forest & $O(1)$ & \cite{BondyMurty2008}\\
\hline
$F=K_3$ & $(1+o(1))\frac{\Delta}{\log\Delta}$ & \cite{Molloy2019}\\
\hline
$F$ a cycle & $(1+o(1))\frac{\Delta}{\log\Delta}$ &
\cite{DaviesKangPirotSereni2020}\\
\hline
$F$ bipartite & $(1+o(1))\frac{\Delta}{\log\Delta}$ &
\cite{AndersonBernshteynDhawan2023}\\
\hline
$F$ $3$-colorable & $O(\frac{\Delta}{\log\Delta})$ &
\cite{DhawanJanzerMethuku2025}\\
\hline
$F$ arbitrary &
$O(\frac{\Delta\log\log\Delta}{\log\Delta})$ &
\cite{Johansson1996,Molloy2019}\\
\hline
\end{tabular}
\end{table}

Hind, Molloy, and Reed~\cite{HindMolloyReed1997} introduced frugal coloring
in order to control repetition of a color inside a neighborhood. In the same paper, they established the following upper bound and
presented a construction due to Alon showing that it is sharp up to
a constant factor. More precisely, for every fixed $\beta\ge1$, there
exist positive constants $C',C$ and $\Delta_0$ such that
every graph $G$ of maximum degree $\Delta\ge\Delta_0$ satisfies
$\chi_\beta(G)\le C\Delta^{1+\frac{1}{\beta}}$.
Moreover, for every integer $\Delta\ge\Delta_0$, there exists a
bipartite graph $G$ of maximum degree $\Delta$ satisfying
$\chi_\beta(G)\ge C'\Delta^{1+\frac{1}{\beta}}$.
Kang and M\"uller~\cite{KangMuller2011} subsequently improved the
lower bound to $(\frac{1}{\beta}-o(1))\Delta^{1+\frac{1}{\beta}}$
for every fixed $\beta\ge1$, and developed further connections
among frugal, acyclic, and star colorings. 

Since the general upper bound for the frugal coloring is sharp up to a constant factor, a
natural next step is to seek improved bounds under sparsity assumptions.
In analogy with the study of proper coloring under subgraph exclusions,
one may ask which fixed forbidden subgraphs $F$ force an asymptotic reduction
in $\chi_\beta(\Delta,F)$. The bipartite lower-bound constructions of Alon and Kang and
M\"uller~\cite{HindMolloyReed1997,KangMuller2011} show that excluding a nonbipartite graph cannot yield such a
reduction. However, excluding a bipartite graph does not always suffice:
when $\beta=1$, the point--line incidence graphs of projective planes
are $C_{4}$-free, but their squares require $\Omega(\Delta^2)$ colors~\cite{HindMolloyReed1997}.
These observations motivate the following problem posed by
Chuet~\cite{Chuet2026}.

\begin{problem}[{\cite[Problem~A]{Chuet2026}}]\label{prob:A}
Fix $\beta\ge1$. For which fixed bipartite graphs $F$ with at least one
edge does $\chi_\beta(\Delta,F)=o(\Delta^{1+\frac{1}{\beta}})$ as
$\Delta\to\infty$?
\end{problem}

Chuet~\cite{Chuet2026} answered \textbf{Problem~\ref{prob:A}} for two
fundamental classes of bipartite obstructions.

\begin{theorem}[{\cite{Chuet2026}}]\label{thm:chuet}
For every fixed $\beta,r,t\ge2$, there exist a constant
$C=C(\beta,r,t)>0$ and $\Delta_0$ such that every graph $G$ of
maximum degree $\Delta\ge\Delta_0$ that is either $K_{\beta,r}$-free
or $C_{2t}$-free satisfies
$\chi_\beta(G)\le
C\frac{\Delta^{1+\frac{1}{\beta}}}{(\log\Delta)^{\frac{1}{\beta}}}$.
The order is best possible even for graphs of arbitrarily large prescribed
girth.
\end{theorem}

In 2011, Aravind and Subramanian~\cite{AravindSubramanian2011}
introduced $(2,\calF)$-avoiding colorings, which include acyclic,
star, and frugal colorings as special cases. In that work and their
subsequent paper~\cite{AravindSubramanian2013}, they established the
following bounds on $\chi_{2,\calF}(\Delta)$.

\begin{theorem}[{\cite{AravindSubramanian2011,AravindSubramanian2013}}]
\label{thm:AS-bounds}
Let $\calF$ be a fixed nonempty family of connected bipartite graphs,
and let $m\ge2$ be the minimum number of edges in a member of $\calF$.
There exist positive constants $C',C$ and $\Delta_0$
such that every graph $G$ of maximum degree $\Delta\ge\Delta_0$
satisfies $\chi_{2,\calF}(G)\le C\Delta^{\frac{m}{m-1}}$.
Moreover, for every integer $\Delta\ge\Delta_0$, there exists a graph
$G$ of maximum degree $\Delta$ satisfying
$\chi_{2,\calF}(G)\ge
C'(\frac{\Delta^m}{\log\Delta})^{\frac{1}{m-1}}$.
\end{theorem}

Chuet also asked whether the upper bound in
\textbf{Theorem~\ref{thm:AS-bounds}} is always of the correct order.

\begin{problem}[{\cite[Problem~B]{Chuet2026}}]\label{prob:B}
Let $H$ be a fixed connected bipartite graph with $m\ge2$ edges. Is
$\chi_{2,H}(\Delta)=\Theta(\Delta^{\frac{m}{m-1}})$ as
$\Delta\to\infty$?
\end{problem}

He also asked whether a fixed girth condition forces an asymptotic reduction for
tree obstructions.

\begin{problem}[{\cite[Problem~C]{Chuet2026}}]\label{prob:C}
Let $H$ be a fixed tree with $m\ge2$ edges. Does there exist an integer
$g\ge3$ such that
$\max\limits_{\Delta(G)\le\Delta,\,g(G)\ge g}\chi_{2,H}(G)
=o(\Delta^{\frac{m}{m-1}})$ as $\Delta\to\infty$?
\end{problem}

In this paper, we address
\textbf{Problems~\ref{prob:A} and~\ref{prob:C}}. To describe the bipartite exclusions
covered by our results, we introduce the following structural parameter.
Let $F$ be a fixed connected bipartite graph with at least one edge
and bipartition classes $X_F$ and $Y_F$. Define
\[
k(F)
=
\min\bigl\{|I|: I\subseteq X_F\text{ or }I\subseteq Y_F,
\ F-I\text{ is a forest}\bigr\}.
\]

Our first main theorem gives a logarithmic improvement in the parameter
range $k(F)\le m-2$.

\begin{theorem}\label{thm:main}
Let $\calF$ be a nonempty family of connected bipartite graphs, and let
$m\ge2$ be the minimum number of edges in a member of $\calF$. Let $F$ be a
fixed connected bipartite graph with at least one edge and suppose that
$k(F)\le m-2$. There exist a constant $C=C(m,F)>0$ and $\Delta_0$ such
that every $F$-free graph $G$ of maximum degree $\Delta\ge\Delta_0$
satisfies
$\chi_{2,\calF}(G)\le
C(\frac{\Delta^m}{\log\Delta})^{\frac{1}{m-1}}$.
\end{theorem}

Since $k(K_{m-1,r})=\min\{m-1,r\}-1\le m-2$ and
$k(C_{2t})=1\le m-2$ whenever $m\ge3$ and $r,t\ge2$,
we obtain the following corollary.

\begin{corollary}\label{cor:cycle}
Let $\calF$ be a nonempty family of connected bipartite graphs, and let
$m\ge3$ be the minimum number of edges in a member of $\calF$.
For every fixed $r,t\ge2$, there exist a constant $C=C(m,r,t)>0$
and $\Delta_0$ such that every graph $G$ of maximum degree
$\Delta\ge\Delta_0$ that is either $K_{m-1,r}$-free or $C_{2t}$-free
satisfies
$\chi_{2,\calF}(G)\le
C(\frac{\Delta^m}{\log\Delta})^{\frac{1}{m-1}}$.
\end{corollary}

Alon and Mohar~\cite{AlonMohar2002} proved that every graph $G$ of
maximum degree $\Delta$ and girth at least seven satisfies
$\chi(G^2)=O(\frac{\Delta^2}{\log\Delta})$.
Combining this result with \textbf{Corollary~\ref{cor:cycle}} gives
the following uniform girth thresholds.

\begin{corollary}\label{cor:girth}
Let $\calF$ be a nonempty family of connected bipartite graphs, and let
$m\ge2$ be the minimum number of edges in a member of $\calF$. There exist
a constant $C=C(m)>0$ and $\Delta_0$ such that every graph $G$ of
maximum degree $\Delta\ge\Delta_0$ and girth at least seven when
$m=2$, or at least five when $m\ge3$, satisfies
$\chi_{2,\calF}(G)\le
C(\frac{\Delta^m}{\log\Delta})^{\frac{1}{m-1}}$.
\end{corollary}

In particular, \textbf{Corollary~\ref{cor:girth}} answers
\textbf{Problem~\ref{prob:C}} affirmatively. The girth thresholds are
best possible. Indeed, taking $\calF=\{K_{1,\beta+1}\}$ gives
$\chi_{2,\calF}(G)=\chi_\beta(G)$. Alon's constructions~\cite{HindMolloyReed1997} require
$\Omega(\Delta^{1+\frac{1}{\beta}})$ colors in every $\beta$-frugal coloring
and have girth six when $\beta=1$ and four when $\beta\ge2$.

The lower bound in \textbf{Theorem~\ref{thm:AS-bounds}} matches
the order of the upper bound in \textbf{Corollary~\ref{cor:girth}},
but does not impose a girth condition.
The following result shows that the upper bound in
\textbf{Corollary~\ref{cor:girth}} remains best possible in order for every
fixed tree obstruction under any fixed prescribed girth. Note that Chuet~\cite{Chuet2026} already observed that his random-graph argument
can be adapted to obtain the following lower bound. We include a complete proof in \textbf{Section~\ref{s:lower}} for completeness.

\begin{theorem}\label{thm:lower-avoiding}
Let $H$ be a fixed tree with $m\ge2$ edges, and let $g\ge3$ be fixed.
There exist a constant $C=C(H,g)>0$ and $\Delta_0$ such that, for
every integer $\Delta\ge\Delta_0$, there exists a graph $G$ of maximum
degree $\Delta$ and girth at least $g$ satisfying
$\chi_{2,H}(G)\ge
C(\frac{\Delta^m}{\log\Delta})^{\frac{1}{m-1}}$.
\end{theorem}

Taking $\calF=\{K_{1,\beta+1}\}$ gives $m=\beta+1$, so
\textbf{Corollary~\ref{cor:cycle}} and
\textbf{Theorem~\ref{thm:lower-avoiding}} generalize
\textbf{Theorem~\ref{thm:chuet}}. More generally,
\textbf{Theorem~\ref{thm:main}} gives a positive answer to
\textbf{Problem~\ref{prob:A}} whenever $k(F)\le\beta-1$.
The complementary range $k(F)\ge m-1$, which is closely related
to \textbf{Problem~\ref{prob:B}}, is discussed in
\textbf{Section~\ref{s:further}}.

We next turn to B-colorings. B-coloring was introduced by Gy\'arf\'as and
S\'ark\"ozy~\cite{GyarfasSarkozy2023} in connection with the
Brown--Erd\H{o}s--S\'os $(7,4)$-conjecture~\cite{BrownErdosSos1973}. They
asked whether every balanced bipartite graph $G=(X,Y)$ with
$|X|=|Y|=n$ and $q_B(G)=n$ must satisfy $|E(G)|=o(n^2)$; a positive answer
would imply the $(7,4)$-conjecture.

Gy\'arf\'as and S\'ark\"ozy~\cite{GyarfasSarkozy2023} observed that
every graph $G$ satisfies $q_B(G)\le\Delta(G)^2$, and that this
bound is sharp since $q_B(K_{\Delta,\Delta})=\Delta^2$.
It is therefore natural to ask which fixed subgraph exclusions
yield a bound of order $o(\Delta^2)$. Since for every fixed nonbipartite
graph $F$, $K_{\Delta,\Delta}$ is $F$-free and
$q_B(K_{\Delta,\Delta})=\Delta^2$, no nonbipartite exclusion can yield
such a bound. We therefore focus on bipartite exclusions.

A strong edge-coloring is in particular a B-coloring, so
$q_B(G)\le\chi'_s(G)$. Bi, Bradshaw, Dhawan, and
Xu~\cite{BiBradshawDhawanXu2026} recently proved that, for every fixed
$t\ge2$, every $K_{t,t}$-free graph of maximum degree $\Delta$
satisfies $\chi'_s(G)\le(1+o(1))\frac{\Delta^2}{\log\Delta}$.
Since every fixed bipartite $F$ is contained in $K_{t,t}$ for some fixed
$t$, every $F$-free graph $G$ of maximum degree $\Delta$ satisfies
$q_B(G)\le(1+o(1))\frac{\Delta^2}{\log\Delta}$
as $\Delta\to\infty$.
Our next theorem improves the bound $O(\frac{\Delta^2}{\log\Delta})$
in terms of $k(F)$.

\begin{theorem}\label{thm:B-k}
Let $F$ be a fixed connected bipartite graph with at least one edge, and put
$k=k(F)$, $h=|V(F)|$, and $s=\min\{|X_F|,|Y_F|\}$. There exist constants
$\eta\in(0,1)$, $C>0$, and $\Delta_0$ such that every $F$-free graph
$G$ of maximum degree $\Delta\ge\Delta_0$ satisfies
\[
q_B(G)\le
\begin{cases}
\Delta+\Delta^{1-\eta}+1, & \text{if }s\le2,\\[4pt]
(4h-2)(\Delta-1)+1, & \text{if }s\ge3\text{ and }k\le1,\\[4pt]
C\frac{\Delta^{2-\frac{1}{k}}}{\log\Delta}, & \text{if }k\ge2.
\end{cases}
\]
\end{theorem}

For $s\le2$, the leading coefficient in
\textbf{Theorem~\ref{thm:B-k}} is best possible in general,
since every B-coloring is a proper edge-coloring and hence
$q_B(G)\ge\Delta$. Since $K_{s-1,\Delta}$ is $F$-free and satisfies
$q_B(K_{s-1,\Delta})=(s-1)\Delta$ for $\Delta\ge s-1$,
the leading coefficient cannot in general be reduced to $1$
when $k(F)\le1$ and $s\ge3$.

Our next result shows that the bound in
\textbf{Theorem~\ref{thm:B-k}} for $k(F)\ge2$ is nearly sharp.

\begin{theorem}\label{thm:B-lower}
Let $k\ge2$ and $t\ge k+1$ be fixed, and suppose
$(k-1)(t-2)>2$. Put
$\gamma_{k,t} = 2-\frac{1}{k}-\frac{k+1}{k(t-1)}$.
There exist a constant $C=C(k,t)>0$ and $\Delta_0$ such that, for
every integer $\Delta\ge\Delta_0$, there exists a $K_{k+1,t}$-free
graph $G$ of maximum degree $\Delta$ satisfying
$q_B(G)\ge C\frac{\Delta^{\gamma_{k,t}}}{\log\Delta}$.
\end{theorem}

For fixed $k\ge2$ and $\varepsilon>0$, choose an integer $t\ge k+1$
such that $(k-1)(t-2)>2$ and $\frac{k+1}{k(t-1)}\le\varepsilon$.
Then $F=K_{k+1,t}$ satisfies $k(F)=k$ and
$\gamma_{k,t}\ge2-\frac{1}{k}-\varepsilon$, so
\textbf{Theorem~\ref{thm:B-lower}} gives the following corollary.

\begin{corollary}\label{cor:B-k-sharp}
For every integer $k\ge2$ and every $\varepsilon>0$, there exist a fixed
connected bipartite graph $F$ with $k(F)=k$, a constant $C>0$, and
$\Delta_0$ such that, for every integer $\Delta\ge\Delta_0$, there
exists an $F$-free graph $G$ of maximum degree $\Delta$ satisfying
$q_B(G)\ge C\frac{\Delta^{2-\frac{1}{k}-\varepsilon}}{\log\Delta}$.
\end{corollary}

The remainder of this paper is organized as follows.
\textbf{Section~\ref{s:auxiliary}} records the forbidden-submatching theorem
of Delcourt and Postle and a bound for coloring graph squares, proves
two counting lemmas, and then recalls the probability inequalities used
in the proofs. \textbf{Section~\ref{s:upper}} proves the upper bounds for $(2,\calF)$-avoiding colorings and B-colorings. Both upper bounds are obtained by reducing coloring problems to $P$-perfect matching problems in auxiliary hypergraphs. The corresponding lower bounds are
proved in \textbf{Section~\ref{s:lower}}. Finally,
in \textbf{Section~\ref{s:further}}, we discuss further work for the
two coloring problems.

\section{Auxiliary Lemmas}\label{s:auxiliary}

We use the following forbidden-submatching theorem of Delcourt and Postle.

\begin{lemma}[{\cite[Theorem~1.16]{DelcourtPostle2024}}]\label{lem:DP}
For all integers $r,g\ge2$ and every real $\theta\in(0,1)$, there exist an
integer $D_\theta\ge0$ and a real number $\xi>0$ such that the following
holds for every $D\ge D_\theta$. Let $\mathcal A$ be a bipartite
$r$-bounded hypergraph with parts $P,Q$ and
$\codeg(\mathcal A)\le D^{1-\theta}$ such that
$d_{\mathcal A}(x)\ge(1+D^{-\xi})D$ for every $x\in P$ and
$d_{\mathcal A}(y)\le D$ for every $y\in Q$. Let $\mathcal C$ be a
$g$-bounded configuration hypergraph for $\mathcal A$ satisfying
$\Delta_i(\mathcal C)\le\xi D^{i-1}\log D$ for every $2\le i\le g$ and
$\Delta_{k,\ell}(\mathcal C)\le D^{k-\ell-\theta}$ for every
$2\le\ell<k\le g$. If
$\codeg_2(\mathcal A,\mathcal C)\le D^{1-\theta}$ and
$\cdeg_2(\mathcal C)\le D^{1-\theta}$, then $\mathcal A$ contains a
$\mathcal C$-avoiding $P$-perfect matching. In fact, $\mathcal A$ contains
at least $D$ pairwise edge-disjoint $\mathcal C$-avoiding $P$-perfect
matchings.
\end{lemma}

We also use the following bound on the chromatic number of the square
of a $d$-degenerate graph.

\begin{lemma}[{\cite{KiersteadYangYi2020}}]\label{lem:square-degenerate}
For every integer $d\ge1$, every $d$-degenerate graph $G$ satisfies
$\chi(G^2)\le(2d-1)\Delta(G)+2d+1$.
\end{lemma}

The following counting lemma is used in the proofs of
\textbf{Theorems~\ref{thm:main} and~\ref{thm:B-k}}.
For a graph $L$ and sets $A,B\subseteq V(L)$, define
$\mathcal E_L(A,B)=\{(a,b)\in A\times B:ab\in E(L)\}$ and
$e_L(A,B)=|\mathcal E_L(A,B)|$.
For $v\in V(L)$ and
$S\subseteq V(L)$, write $d_L(v,S)=|N_L(v)\cap S|$.

\begin{lemma}\label{lem:one-sided-forest}
Let $H$ be a bipartite graph with a fixed bipartition
$V(H)=X_H\cup Y_H$, where $X_H\cap Y_H=\emptyset$ and $X_H\ne\emptyset$.
Let $Z\subseteq Y_H$ satisfy $|Z|=r$ and $H-Z$ is a forest.
There is a constant $C>0$ such that the following holds.
If $L$ is a graph and nonempty sets $A,B\subseteq V(L)$ admit no injective map
$\phi:V(H)\to V(L)$ satisfying
$\phi(X_H)\subseteq A$, $\phi(Y_H)\subseteq B$, and
$\phi(x)\phi(y)\in E(L)$ for every $xy\in E(H)$, then
$e_L(A,B)\le C(|A|+|B|\,|A|^{\frac{r}{r+1}})$.
For $r=0$, one may take $C=2|V(H)|$.
\end{lemma}

\begin{proof}
Put $h=|V(H)|$ and $M=e_L(A,B)$. If $M=0$, the assertion is immediate,
so assume $M>0$. In particular, $A$ and $B$ are nonempty.
We argue by induction on $r$.

Suppose first that $r=0$. Assume, for a contradiction, that
$M>2h(|A|+|B|)$. Starting with $A_0=A$ and $B_0=B$, perform the
following deletions. If $a\in A_i$ satisfies $d_L(a,B_i)<2h$, set
$A_{i+1}=A_i\setminus\{a\}$ and $B_{i+1}=B_i$. Otherwise, if
$b\in B_i$ satisfies $d_L(b,A_i)<2h$, set $A_{i+1}=A_i$ and
$B_{i+1}=B_i\setminus\{b\}$. Stop when neither deletion is possible.

Each step removes fewer than $2h$ pairs from
$\mathcal E_L(A_i,B_i)$, and there are at most $|A|+|B|$ steps.
Consequently the final sets $A_*,B_*$ satisfy $e_L(A_*,B_*)>0$,
$d_L(a,B_*)\ge2h$ for every $a\in A_*$, and
$d_L(b,A_*)\ge2h$ for every $b\in B_*$. In particular,
$|A_*|,|B_*|\ge2h$.

Root each component of the forest $H$, including its isolated vertices.
Order $V(H)=\{v_1,\ldots,v_h\}$ so that the root of each component
precedes its other vertices and every nonroot vertex precedes all its
children. Every nonroot $v_i$ then has exactly one earlier neighbor,
its parent $v_{p(i)}$. Put $S_i=A_*$ if $v_i\in X_H$ and $S_i=B_*$
if $v_i\in Y_H$. We construct $\phi(v_i)$ successively.
After choosing $\phi(v_1),\ldots,\phi(v_{i-1})$, let
$U_i=\{\phi(v_1),\ldots,\phi(v_{i-1})\}$.
If $v_i$ is a root, choose $\phi(v_i)\in S_i\setminus U_i$.
Otherwise choose
$\phi(v_i)\in (N_L(\phi(v_{p(i)}))\cap S_i)\setminus U_i$.
In either case the set of available vertices has size at least
$2h-(i-1)>0$. Thus $\phi$ is injective, maps the two classes of $H$
into $A_*,B_*$, respectively, and preserves every edge of $H$.
This contradicts the hypothesis and proves the assertion for $r=0$.

Now let $r\ge1$ and choose $z\in Z$. Put $H'=H-z$ and
$Z'=Z\setminus\{z\}$, so that $H'-Z'=H-Z$ is a forest.
Let $B^+=\{b\in B:d_L(b,A)>0\}$. Since $M>0$, the set $B^+$ is
nonempty. For each $b\in B^+$, put
$A_b=N_L(b)\cap A$, $B_b=B\setminus\{b\}$, and
$L_b=L[A_b\cup B_b]$.

There is no injective map $\phi:V(H')\to V(L_b)$ with
$\phi(X_H)\subseteq A_b$, $\phi(Y_H\setminus\{z\})\subseteq B_b$,
and $\phi(x)\phi(y)\in E(L_b)$ for every $xy\in E(H')$.
Indeed, $b\notin A_b\cup B_b$, since $L$ has no loops.
Extending such a map by $\phi(z)=b$ would remain injective.
Every neighbor of $z$ in $H$ belongs to $X_H$, whose image lies in
$A_b\subseteq N_L(b)$, so the extension would preserve all edges of $H$.
The sets in this extension are illustrated in
\textbf{Figure~\ref{fig:one-sided-counting}}.

If $B_b=\emptyset$, then $e_{L_b}(A_b,B_b)=0$.
Otherwise, both $A_b$ and $B_b$ are nonempty, and we apply the
induction hypothesis to $H'$ with classes $X_H$ and
$Y_H\setminus\{z\}$ and deletion set $Z'$.
Consequently, with a constant $C'>0$ depending only on $H'$, for every $b\in B^+$ we have
\begin{equation}\label{eq:forest-induction}
e_{L_b}(A_b,B_b)
\le C'\bigl(d_L(b,A)+|B|\,d_L(b,A)^{1-\frac{1}{r}}\bigr).
\end{equation}
To sum \eqref{eq:forest-induction} over $b\in B^+$, define
\[
\begin{aligned}
\mathcal P
&=\{(b,(a,b')):b\in B^+,\ (a,b')\in\mathcal E_{L_b}(A_b,B_b)\},\\
\mathcal T
&=\{(a,b,b'):a\in A,\ b,b'\in N_L(a)\cap B,\ b\ne b'\}.
\end{aligned}
\]
The map $\Psi:\mathcal P\to\mathcal T$ given by
$\Psi(b,(a,b'))=(a,b,b')$ is a bijection. In fact, membership in
$\mathcal P$ gives $ab,ab'\in E(L)$ and $b'\ne b$.
Conversely, if $(a,b,b')\in\mathcal T$, then $b\in B^+$,
$a\in A_b$, $b'\in B_b$, and $ab'\in E(L_b)$, so
$(b,(a,b'))$ is its unique preimage. Therefore
\begin{equation}\label{eq:forest-double-count}
\begin{aligned}
\sum_{a\in A}d_L(a,B)(d_L(a,B)-1)
&=|\mathcal T|=|\mathcal P|
=\sum_{b\in B^+}e_{L_b}(A_b,B_b)\\
&\le C'M+C'|B|\sum_{b\in B^+}d_L(b,A)^{1-\frac{1}{r}}.
\end{aligned}
\end{equation}
Here $\sum_{b\in B^+}d_L(b,A)=M$.
For $r>1$, concavity gives
\begin{equation}\label{eq:forest-power-sum}
\sum_{b\in B^+}d_L(b,A)^{1-\frac{1}{r}}
\le |B^+|^{\frac{1}{r}}M^{1-\frac{1}{r}}
\le |B|^{\frac{1}{r}}M^{1-\frac{1}{r}}.
\end{equation}
For $r=1$, \eqref{eq:forest-power-sum} holds because
$\sum_{b\in B^+}d_L(b,A)^0=|B^+|\le|B|$.
By the Cauchy--Schwarz inequality,
$M^2\le |A|\sum_{a\in A}d_L(a,B)^2$.
Together with \eqref{eq:forest-double-count} and
\eqref{eq:forest-power-sum}, this yields
\[
\frac{M^2}{|A|}
\le (C'+1)M+C'|B|^{1+\frac{1}{r}}M^{1-\frac{1}{r}}.
\]
If $(C'+1)M\ge\frac{M^2}{2|A|}$, then $M\le2(C'+1)|A|$.
Otherwise,
$\frac{M^2}{2|A|}\le C'|B|^{1+\frac{1}{r}}M^{1-\frac{1}{r}}$, and hence
$M\le(2C')^{\frac{r}{r+1}}|B|\,|A|^{\frac{r}{r+1}}$.
Taking $C\ge\max\{2(C'+1),(2C')^{\frac{r}{r+1}}\}$ proves the
claim. Since $H$ has only finitely many bipartitions and subsets $Z$,
the constant can be chosen to depend only on $H$.
\end{proof}

\begin{figure}[htbp]
\centering
\begin{tikzpicture}[font=\small]
\tikzset{pt/.style={circle,fill=black,inner sep=1.5pt},
  edge/.style={semithick},
  lab/.style={fill=none,inner sep=1.5pt}}
\draw[thick] (-3.3,0) ellipse (2.3 and 2.8);
\draw[thick] (3.3,0) ellipse (2.3 and 2.8);
\node at (-3.3,3.1) {$A$};
\node at (3.3,3.1) {$B$};
\filldraw[fill=black!5,draw=black!60] (-3.3,-.25) ellipse (1.98 and 2.18);
\filldraw[fill=black!5,draw=black!60] (3.3,-.68) ellipse (1.98 and 1.78);
\node[lab] at (-3.3,1.3) {$A_b=N_L(b)\cap A$};
\node[lab] at (3.3,.6) {$B_b=B\setminus\{b\}$};
\draw (-3.3,-.45) ellipse (1.3 and 1.15);
\node[lab] at (-3.3,-1.28) {$\phi(X_H)$};
\draw (3.3,-.82) ellipse (1.55 and 1.08);
\node[lab] at (3.3,-1.50) {$\phi(Y_H\setminus\{z\})$};
\draw[dashed] (3.8,-.66) ellipse (.85 and .58);
\node[lab,font=\scriptsize] at (3.8,-.88) {$\phi(Z\setminus\{z\})$};
\node[pt] (x1) at (-3.7,.04) {};
\node[pt] (x2) at (-2.85,-.2) {};
\node[pt] (x3) at (-3.45,-.67) {};
\node[pt] (y1) at (2.45,-.49) {};
\node[pt] (y2) at (2.95,-1.0) {};
\node[pt] (y3) at (3.8,-.45) {};
\node[pt,label=above:{$b=\phi(z)$}] (b) at (3.3,1.95) {};
\draw[edge] (x1)--(y1);
\draw[edge] (x2)--(y2);
\draw[edge] (x2)--(y3);
\draw[edge] (x3)--(y2);
\draw[dashed,edge] (b)--(x1);
\draw[dashed,edge] (b)--(x2);
\draw[dashed,edge] (b)--(x3);
\node[lab] at (0,-2.85) {$L_b=L[A_b\cup B_b]$};
\end{tikzpicture}
\caption{Extending an embedding of $H-z$ by $\phi(z)=b$.
The shaded regions are $A_b$ and $B_b$, whose union induces $L_b$.
Every vertex of $A_b$ is adjacent to $b$. The sets $A,B$ may overlap,
although drawn separately.}
\label{fig:one-sided-counting}
\end{figure}
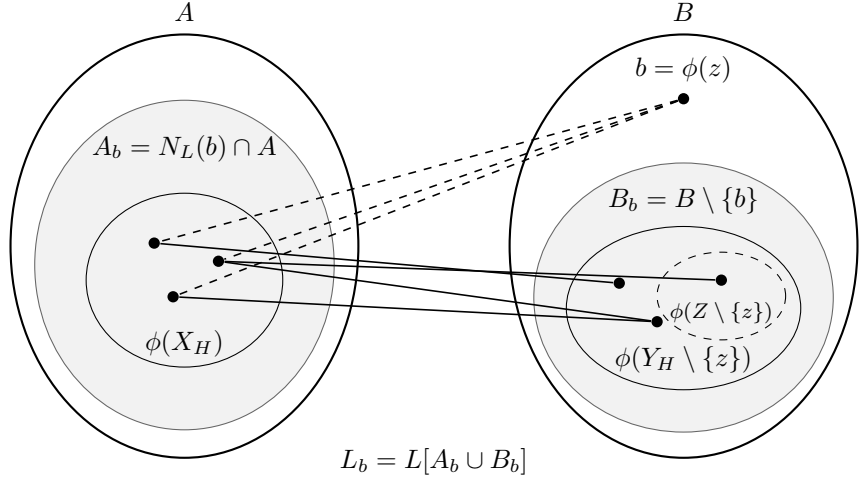

We next establish another counting lemma needed in the proof of
\textbf{Theorem~\ref{thm:B-k}}.
For an edge $e\in E(G)$, let
$\Omega_G(e)=\{f\in E(G):e\text{ and }f\text{ are opposite edges of a }C_4\}$.
When $G$ is fixed, write $\Omega(e)=\Omega_G(e)$.

\begin{lemma}\label{lem:B-local-k}
Let $F$ be a fixed connected bipartite graph with at least one edge, put
$k=k(F)$, and let $G$ be an $F$-free graph with maximum degree $\Delta$.
For $e\in E(G)$ and $v\in V(G)$, put
$\rho(e,v)=|\{f\in\Omega(e):v\in f\}|$.
For every $e\in E(G)$ and $v\in V(G)$, we have
$\rho(e,v)\le2(\Delta-1)$, and the following bounds hold:
\begin{enumerate}
\item\label{item:B-local-linear} if $k\le1$, then
$|\Omega(e)|\le4(|V(F)|-1)(\Delta-1)$;
\item\label{item:B-local-power} if $k\ge2$, then $|\Omega(e)|=O(\Delta^{2-\frac{1}{k}})$;
\item if $k\ge2$ and $e,f\in E(G)$ are distinct, then
$|\Omega(e)\cap\Omega(f)|=O(\Delta^{2-\frac{1}{k-1}})$.
\end{enumerate}
Moreover, if $G$ is $K_{2,t}$-free for an integer $t\ge2$, then
$|\Omega(e)|\le(t-1)\Delta$, $\rho(e,v)\le2(t-1)$, and
$|\Omega(e)\cap\Omega(f)|\le4(t-1)^2$ for all distinct edges $e,f$.
\end{lemma}

\begin{proof}
There is nothing to prove when $E(G)=\emptyset$, so assume $\Delta\ge1$.
Put $h=|V(F)|$. By the definition of $k(F)$, there are a bipartition
$V(F)=X_F\cup Y_F$ with $X_F\cap Y_F=\emptyset$ and a set
$I\subseteq Y_F$ such that
$|I|=k$ and $F-I$ is a forest. If $k\ge1$, write
$I=\{x_1,\ldots,x_k\}$. If $k=0$, choose any $x_1\in Y_F$;
this class is nonempty because $F$ has an edge.

Fix $e=uv$. If $\Omega(e)=\emptyset$, all bounds involving $e$ are
immediate, so assume $\Omega(e)\ne\emptyset$.
Put $A_e=N_G(v)\setminus\{u\}$ and $B_e=N_G(u)\setminus\{v\}$.
Both sets are nonempty and contained in
$V(G)\setminus\{u,v\}$.
The map $\pi_e:\mathcal E_G(A_e,B_e)\to\Omega(e)$ given by
$\pi_e(x,y)=xy$ is surjective. Indeed, each pair $(x,y)$ in its domain
gives the four distinct vertices of the cycle $vxyuv$.
Conversely, if $f\in\Omega(e)$, its endpoints can be named $x,y$ so that
$vxyuv$ is a $4$-cycle, and then $(x,y)\in\mathcal E_G(A_e,B_e)$.
It follows that $|\Omega(e)|\le e_G(A_e,B_e)$.

Set $H=F-x_1$, with bipartition $X_H=X_F$ and
$Y_H=Y_F\setminus\{x_1\}$. There is no injective map
$\phi:V(H)\to V(G)$ with $\phi(X_H)\subseteq A_e$,
$\phi(Y_H)\subseteq B_e$, and
$\phi(x)\phi(y)\in E(G)$ for every $xy\in E(H)$.
Such a map would avoid $v$ and could be extended by $\phi(x_1)=v$:
every neighbor of $x_1$ in $F$ belongs to $X_F$, whose image is contained
in $A_e\subseteq N_G(v)$.

If $k\le1$, then $H$ is a forest. Applying
\textbf{Lemma~\ref{lem:one-sided-forest}} to $H$, $L=G$,
$A=A_e$, $B=B_e$, and $Z=\emptyset$ gives
\[
|\Omega(e)|
\le e_G(A_e,B_e)
\le2(h-1)(|A_e|+|B_e|)
\le4(h-1)(\Delta-1).
\]
If $k\ge2$, take $Z=I\setminus\{x_1\}\subseteq Y_H$, so that
$|Z|=k-1$ and $H-Z=F-I$ is a forest.
\textbf{Lemma~\ref{lem:one-sided-forest}}, with $r=k-1$, gives
\[
|\Omega(e)|
\le e_G(A_e,B_e)
\le C\bigl(|A_e|+|B_e|\,|A_e|^{\frac{k-1}{k}}\bigr)
=O(\Delta^{2-\frac{1}{k}}).
\]
This proves parts~\ref{item:B-local-linear} and~\ref{item:B-local-power}
of \textbf{Lemma~\ref{lem:B-local-k}}.

Next fix $z\in V(G)$. If $z\in\{u,v\}$, then $\rho(e,z)=0$.
Otherwise each edge $zw\in\Omega(e)$ satisfies
$w\in N_G(z)\cap(N_G(u)\setminus\{v\})$ or
$w\in N_G(z)\cap(N_G(v)\setminus\{u\})$.
Since distinct edges incident to $z$ have distinct other endpoints,
\begin{equation}\label{eq:B-rho}
\begin{aligned}
\rho(e,z)
&\le |N_G(z)\cap(N_G(u)\setminus\{v\})|
   +|N_G(z)\cap(N_G(v)\setminus\{u\})|\\
&\le(d_G(u)-1)+(d_G(v)-1)
\le2(\Delta-1).
\end{aligned}
\end{equation}

We now count common members of $\Omega(e)$ and $\Omega(f)$ for distinct
edges $e,f$. Put $S=e\cup f$ and
$\mathcal I=\{(a,c)\in e\times f:a\ne c\}$.
For $(a,c)\in\mathcal I$, let $a'$ and $c'$ be determined by
$e=\{a,a'\}$ and $f=\{c,c'\}$, and define
\[
A_{a,c}=(N_G(a)\cap N_G(c))\setminus S,
\qquad
B_{a,c}=(N_G(a')\cap N_G(c'))\setminus S.
\]
Let
$\mathcal D=\{(a,c,x,y):(a,c)\in\mathcal I,\
(x,y)\in\mathcal E_G(A_{a,c},B_{a,c})\}$.
The map $\pi:\mathcal D\to\Omega(e)\cap\Omega(f)$ defined by
$\pi(a,c,x,y)=xy$ is surjective.
Every element of $\mathcal D$ gives the two cycles
$axy a'a$ and $cxy c'c$, so its image belongs to the stated range.
Conversely, take $g=xy\in\Omega(e)\cap\Omega(f)$.
The two witnessing cycles give endpoints $a,a'$ of $e$ and $c,c'$ of
$f$ with $ax,cx,a'y,c'y\in E(G)$.
At least one of $a\ne c$ and $a'\ne c'$ holds, since $e\ne f$.
Naming the endpoints $x,y$ so that $a\ne c$, we obtain
$(a,c,x,y)\in\mathcal D$. Thus
\begin{equation}\label{eq:B-common-count}
|\Omega(e)\cap\Omega(f)|
\le\sum_{(a,c)\in\mathcal I}e_G(A_{a,c},B_{a,c}).
\end{equation}

Suppose $k\ge2$. Terms in \eqref{eq:B-common-count} with
$A_{a,c}=\emptyset$ or $B_{a,c}=\emptyset$ are zero.
Fix any remaining $(a,c)\in\mathcal I$.
Put $H=F-\{x_1,x_2\}$, $X_H=X_F$,
$Y_H=Y_F\setminus\{x_1,x_2\}$, and
$Z=I\setminus\{x_1,x_2\}$. Then $Z\subseteq Y_H$,
$|Z|=k-2$, and $H-Z=F-I$ is a forest.
An injective map $\phi:V(H)\to V(G)$ preserving the edges of $H$
and satisfying $\phi(X_H)\subseteq A_{a,c}$ and
$\phi(Y_H)\subseteq B_{a,c}$ would extend to a copy of $F$ by
$\phi(x_1)=a$ and $\phi(x_2)=c$.
Indeed, the original image avoids $S$, the vertices $a,c$ are distinct,
and every vertex of $\phi(X_H)$ is adjacent to both $a$ and $c$.
Thus no such map exists, and
\textbf{Lemma~\ref{lem:one-sided-forest}}, with
$L=G$, $A=A_{a,c}$, $B=B_{a,c}$, and $r=k-2$, gives
\[
e_G(A_{a,c},B_{a,c})
\le C\bigl(|A_{a,c}|+|B_{a,c}|\,|A_{a,c}|^{\frac{k-2}{k-1}}\bigr)
=O(\Delta^{2-\frac{1}{k-1}}).
\]
Since $|\mathcal I|\le4$, \eqref{eq:B-common-count} proves the third
assertion.

Finally, suppose that $G$ is $K_{2,t}$-free. For distinct vertices
$x,y\in V(G)$, we have $|N_G(x)\cap N_G(y)|\le t-1$;
otherwise $x,y$ and any $t$ of their common neighbors give a copy of
$K_{2,t}$. For $e=uv$ and $x\in A_e$, we have $x\ne u$ and
$N_G(x)\cap B_e\subseteq N_G(x)\cap N_G(u)$. Consequently
\[
|\Omega(e)|
\le e_G(A_e,B_e)
=\sum_{x\in A_e}|N_G(x)\cap B_e|
\le(t-1)|A_e|
\le(t-1)\Delta.
\]
If $z\notin\{u,v\}$, the two terms in \eqref{eq:B-rho} are at most
$|N_G(z)\cap N_G(u)|$ and $|N_G(z)\cap N_G(v)|$, respectively.
Each is at most $t-1$, proving $\rho(e,z)\le2(t-1)$;
the case $z\in\{u,v\}$ again gives zero.

For distinct $e,f$, use the sets in \eqref{eq:B-common-count}.
For each $(a,c)\in\mathcal I$, the inequality $a\ne c$ gives
$|A_{a,c}|\le t-1$.
For every $x\in A_{a,c}$, we have $x\notin S$, and hence $x\ne a'$.
Since $B_{a,c}\subseteq N_G(a')$, it follows that
$|N_G(x)\cap B_{a,c}|\le|N_G(x)\cap N_G(a')|\le t-1$.
Therefore
$e_G(A_{a,c},B_{a,c})
=\sum_{x\in A_{a,c}}|N_G(x)\cap B_{a,c}|
\le(t-1)|A_{a,c}|\le(t-1)^2$.
Summing over $|\mathcal I|\le4$ in \eqref{eq:B-common-count} gives
$|\Omega(e)\cap\Omega(f)|\le4(t-1)^2$, as required.
\end{proof}

We also use the following two probability inequalities.

\begin{lemma}[{\cite[Theorem~8.1.1]{AlonSpencer2016}}]\label{lem:janson}
Let $\Omega_p$ be a random subset of a finite set $\Omega$, obtained by
including each element independently with probability $p\in[0,1]$.
For a finite family $\{A_i:i\in I\}$ of nonempty subsets of $\Omega$,
put $\mu=\sum_{i\in I}p^{|A_i|}$ and
$\Gamma=\sum_{\substack{i,j\in I,\ i\ne j\\A_i\cap A_j\ne\emptyset}}
p^{|A_i\cup A_j|}$.
If $\Gamma\le\mu$, then
$\mathbb P(A_i\nsubseteq\Omega_p\text{ for every }i\in I)
\le\exp(-\frac{\mu}{2})$.
\end{lemma}

\begin{lemma}[{\cite[Appendix~A.1]{AlonSpencer2016}}]\label{lem:chernoff}
Let $X\sim\operatorname{Bin}(n,p)$, where $n$ is a positive integer and
$p\in[0,1]$, and put $\mu=np$. For every $\delta\in(0,1]$,
\begin{itemize}
    \item $\mathbb P(X\le(1-\delta)\mu) \le\exp(-\frac{\delta^2\mu}{2}),$
    \item $\mathbb P(X\ge(1+\delta)\mu) \le\exp(-\frac{\delta^2\mu}{3}).$
\end{itemize}
In particular,
$\mathbb P(X\le\frac{\mu}{2})\le e^{-\frac{\mu}{8}}$ and
$\mathbb P(X\ge2\mu)\le e^{-\frac{\mu}{3}}$.
\end{lemma}

\section{Upper Bounds}\label{s:upper}

\subsection{Proof of Theorem~\ref{thm:main}}\label{ss:upper-avoiding}

Let $\mathcal B_m$ be the family of all connected bipartite graphs
with exactly $m$ edges, and fix bipartition classes $X_J$ and $Y_J$
for every $J\in\mathcal B_m$. Since every connected graph with at least $m$ edges contains a connected
subgraph with exactly $m$ edges, every $(2,\mathcal B_m)$-avoiding coloring
is $(2,\calF)$-avoiding.
We therefore construct a $(2,\mathcal B_m)$-avoiding coloring of $G$.

For $m=2$, we have $k(F)=0$, so $F$ is a tree. If $F=K_2$,
then $G$ is edgeless. Otherwise $G$ is $(|V(F)|-2)$-degenerate,
since every graph of minimum degree at least $|V(F)|-1$ contains $F$.
By \textbf{Lemma~\ref{lem:square-degenerate}},
$\chi(G^2)=O(\Delta)=O(\frac{\Delta^2}{\log\Delta})$.
A proper coloring of $G^2$ contains no bichromatic copy of the
unique member $P_3$ of $\mathcal B_2$, proving the assertion.

Assume now that $m\ge3$. Put $k=k(F)$ and choose bipartition classes
$A_F,B_F$ of $F$ together with a set
$I=\{x_1,\ldots,x_k\}\subseteq B_F$ such that $T=F-I$ is a forest. Define
\[
\varepsilon
=
\frac{m-k-1}{8m(m-1)^2},
\qquad
\theta
=
\frac{\varepsilon(m-1)}{2m}.
\]
Then $\theta\in(0,1)$. Put $\alpha_1=\Delta$ and, for
$2\le j\le m-1$, put
$\alpha_j=\Delta^{\frac{m-j}{m-1}-\varepsilon}$.
For every nonempty set $W\subseteq V(G)$, write
$N_G(W)=\bigcap_{w\in W}N_G(w)$.
For $2\le j\le m-1$, an independent $j$-set $S$ is called \emph{special}
if $|N_G(S)|>\alpha_j$ and $S$ contains no special set of smaller size.
A special set of size two is called a \emph{special pair}.

Fix $2\le j\le m-1$, and let
$U=\{u_1,\ldots,u_{j-1}\}\subseteq V(G)$ be independent and contain
no special set. Define
$\Sigma_j(U)=\{w\in V(G)\setminus U:U\cup\{w\}
\text{ is a special }j\text{-set}\}$, and put
$A=N_G(U)$ and $B=\Sigma_j(U)$.
We estimate $|\Sigma_j(U)|$. If $B=\emptyset$, there is nothing to prove,
so assume $B\ne\emptyset$. Then $A\ne\emptyset$, since
$|N_G(U\cup\{w\})|>\alpha_j$ for every $w\in B$.
For $j=2$, we have $|A|\le\Delta=\alpha_1$.
For $j\ge3$, the set $U$ is independent and contains no smaller special
set, so $|A|>\alpha_{j-1}$ would make $U$ special. Thus
$|A|\le\alpha_{j-1}$ in both cases.

The sets $U,A,B$ are pairwise disjoint. Indeed, $U\cap A=\emptyset$
because $G$ has no loops, and $U\cap B=\emptyset$ by definition.
Also, every vertex of $A$ is adjacent to every vertex of $U$, whereas
every vertex of $B$ is nonadjacent to every vertex of $U$, so
$A\cap B=\emptyset$. Let $L$ be the bipartite graph with classes $A,B$
and edge set $E(L)=\{ab\in E(G):a\in A,\ b\in B\}$.
For every $w\in B$, $d_L(w)=|N_G(U\cup\{w\})|>\alpha_j$, so
\begin{equation}\label{eq:special-edge-lower}
e(L)=\sum_{w\in B}d_L(w)>\alpha_j|B|.
\end{equation}
The construction is illustrated in \textbf{Figure~\ref{fig:special-bipartite}}.

Put $t=\min\{j-1,k\}$, $I_t=\{x_1,\ldots,x_t\}$,
$H_j=F-I_t$, $Z_j=I\setminus I_t$, and $r=k-t$, where
$I_0=\emptyset$. Then $H_j-Z_j=T$ is a forest.
There is no injective map $\phi:V(H_j)\to V(L)$ satisfying
$\phi(A_F)\subseteq A$, $\phi(B_F\setminus I_t)\subseteq B$, and
$\phi(x)\phi(y)\in E(L)$ for every $xy\in E(H_j)$.
Otherwise define $\psi:V(F)\to V(G)$ by
$\psi(x_i)=u_i$ for $1\le i\le t$ and
$\psi(v)=\phi(v)$ for $v\in V(H_j)$.
Since $U\cap(A\cup B)=\emptyset$, this map is injective.
The edges of $H_j$ are preserved by $\phi$.
Every remaining edge of $F$ joins a vertex $x_i\in I_t$ to a vertex
of $A_F$, and its image is an edge because every vertex of $U$ is
adjacent to every vertex of $A$. Thus $\psi$ embeds $F$ into $G$,
a contradiction.
Applying \textbf{Lemma~\ref{lem:one-sided-forest}} with
$H=H_j$, $X_H=A_F$, $Y_H=B_F\setminus I_t$, and $Z=Z_j$ gives
\begin{equation}\label{eq:special-edge-upper}
e(L)\le C(|A|+|B|\,|A|^{\frac{r}{r+1}}).
\end{equation}

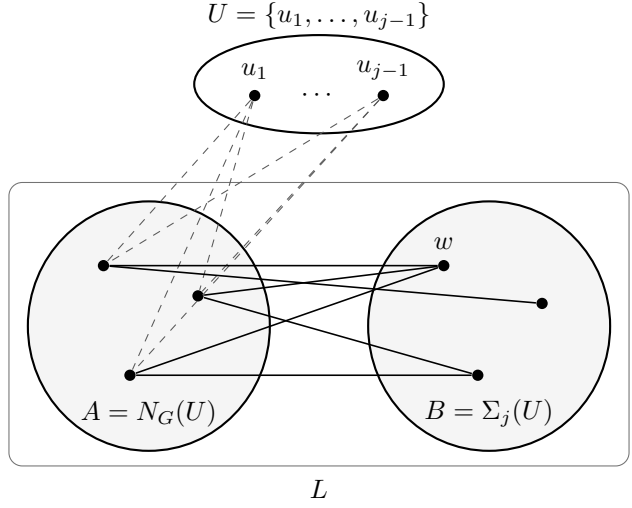
\begin{figure}[htbp]
\centering
\begin{tikzpicture}[font=\small]
\tikzset{pt/.style={circle,fill=black,inner sep=1.5pt},
  lab/.style={fill=none,inner sep=2pt}}
\draw[rounded corners=5pt,black!55] (-4.1,-2.3) rectangle (4.1,1.45);
\filldraw[fill=black!4,thick] (-2.25,-.45) ellipse (1.6 and 1.65);
\filldraw[fill=black!4,thick] (2.25,-.45) ellipse (1.6 and 1.65);
\draw[thick] (0,2.75) ellipse (1.65 and .65);
\node[lab] at (0,3.65) {$U=\{u_1,\ldots,u_{j-1}\}$};
\node[pt,label=above:$u_1$] (u1) at (-.85,2.6) {};
\node at (0,2.6) {$\cdots$};
\node[pt,label=above:$u_{j-1}$] (u2) at (.85,2.6) {};
\node[pt] (a1) at (-2.85,.35) {};
\node[pt] (a2) at (-1.6,-.05) {};
\node[pt] (a3) at (-2.5,-1.1) {};
\node[pt,label=above:$w$] (b1) at (1.65,.35) {};
\node[pt] (b2) at (2.95,-.15) {};
\node[pt] (b3) at (2.1,-1.1) {};
\foreach \u in {u1,u2}{
  \foreach \a in {a1,a2,a3}{
    \draw[dashed,black!65] (\u)--(\a);
  }
}
\draw[semithick] (a1)--(b1);
\draw[semithick] (a2)--(b1);
\draw[semithick] (a3)--(b1);
\draw[semithick] (a1)--(b2);
\draw[semithick] (a2)--(b3);
\draw[semithick] (a3)--(b3);
\node[lab] at (-2.25,-1.6) {$A=N_G(U)$};
\node[lab] at (2.25,-1.6) {$B=\Sigma_j(U)$};
\node[lab] at (0,-2.6) {$L$};
\end{tikzpicture}
\caption{The graph $L$ consists of the edges of $G$ between $A$ and $B$.
All edges between $U$ and $A$ are present, and no edge joins $U$ to $B$.
For $w\in B$, $N_L(w)=N_G(U\cup\{w\})$.}
\label{fig:special-bipartite}
\end{figure}

If $j=2$ and $k\ge1$, then $r=k-1$. Since
$\frac{m-k-1}{k(m-1)}-\varepsilon>0$, we have
$\frac{\alpha_2}{\alpha_1^{\frac{r}{r+1}}}
=\Delta^{\frac{m-k-1}{k(m-1)}-\varepsilon}\to\infty$.
If $j\ge3$ and $j-1\le k$, then $r=k-j+1$. Since
$\frac{m-k-1}{m-1}-\varepsilon>0$, we obtain
$
\frac{\alpha_j}{\alpha_{j-1}^{\frac{r}{r+1}}}
=
\Delta^{\frac{1}{k-j+2}
(\frac{m-k-1}{m-1}-\varepsilon)}
\longrightarrow\infty.
$
If $j-1>k$, then $r=0$ and
$\frac{m-j}{m-1}-\varepsilon
\ge\frac{1}{m-1}-\varepsilon>0$, so
$\frac{\alpha_j}{\alpha_{j-1}^{\frac{r}{r+1}}}=\alpha_j\to\infty$.
Consequently, for sufficiently large $\Delta$,
$C|A|^{\frac{r}{r+1}}
\le C\alpha_{j-1}^{\frac{r}{r+1}}\le\frac{\alpha_j}{2}$.
Substituting this inequality and $|A|\le\alpha_{j-1}$ into
\eqref{eq:special-edge-lower} and \eqref{eq:special-edge-upper} yields
$\alpha_j|B|<e(L)\le
C\alpha_{j-1}+\frac{\alpha_j}{2}|B|$.
After increasing $C$, we obtain
\begin{equation}\label{eq:special-extension-bound}
|\Sigma_j(U)|\le C\frac{\alpha_{j-1}}{\alpha_j}.
\end{equation}

We use \eqref{eq:special-extension-bound} to count special sets
containing prescribed vertices.
For $W\subseteq V(G)$, write
$\mathcal S_j(W)=\{S\subseteq V(G):S\text{ is a special }j\text{-set},
\ W\subseteq S\}$.
Suppose first that $W$ is independent and contains no special set, with
$1\le b=|W|\le j-2$, and put $a=j-b\ge2$.
Define
$\mathcal P_j(W)=\{(S,z):S\in\mathcal S_j(W),\ z\in N_G(S)\}$.
Every $S\in\mathcal S_j(W)$ has more than $\alpha_j$ common neighbors,
so $\alpha_j|\mathcal S_j(W)|\le|\mathcal P_j(W)|$.

To bound $|\mathcal P_j(W)|$, consider triples $(S,z,w)$ with
$(S,z)\in\mathcal P_j(W)$ and $w\in S\setminus W$.
There are exactly $a|\mathcal P_j(W)|$ such triples.
For each triple put $R=S\setminus(W\cup\{w\})$.
The map $(S,z,w)\mapsto(z,R,w)$ is injective, since
$S=W\cup R\cup\{w\}$.
Its image satisfies $z\in N_G(W)$,
$R\in\binom{N_G(z)\setminus W}{a-1}$, and
$w\in\Sigma_j(W\cup R)\cap N_G(z)$.
Moreover, $W\cup R$ is an independent $(j-1)$-set containing no special
set, since it is a proper subset of $S$.
There are at most $|N_G(W)|$ choices for $z$, at most
$\Delta^{a-1}$ choices for $R$, and at most
$C\frac{\alpha_{j-1}}{\alpha_j}$ choices for $w$ by
\eqref{eq:special-extension-bound}. Hence
\[
a\alpha_j|\mathcal S_j(W)|
\le a|\mathcal P_j(W)|
\le C|N_G(W)|\Delta^{a-1}
\frac{\alpha_{j-1}}{\alpha_j}.
\]
Using $a\ge1$ gives
\begin{equation}\label{eq:special-incidence-bound}
|\mathcal S_j(W)|
\le C|N_G(W)|\Delta^{a-1}
\frac{\alpha_{j-1}}{\alpha_j^2}.
\end{equation}

Now fix $u\in V(G)$. For $j=2$,
\eqref{eq:special-extension-bound} gives
$|\mathcal S_2(\{u\})|=|\Sigma_2(\{u\})|
\le C\Delta^{\frac{1}{m-1}+\varepsilon}
\le C\Delta^{1+\varepsilon}$.
For $3\le j\le m-1$, apply \eqref{eq:special-incidence-bound}
with $W=\{u\}$ and $a=j-1$. Since $|N_G(u)|\le\Delta$ and
$j+1-m\le0$,
\[
|\mathcal S_j(\{u\})|
\le C\Delta^{j-1}\frac{\alpha_{j-1}}{\alpha_j^2}
=
C\Delta^{j-1+\frac{j+1-m}{m-1}+\varepsilon}
\le C\Delta^{j-1+\varepsilon}.
\]
Thus
\begin{equation}\label{eq:special-one-vertex}
|\mathcal S_j(\{u\})|\le C\Delta^{j-1+\varepsilon}
\qquad(2\le j\le m-1).
\end{equation}

Finally, let $U$ be an $\ell$-set with
$2\le\ell<j\le m-1$, and put $a=j-\ell$.
If $U$ contains an edge or a special set, then
$\mathcal S_j(U)=\emptyset$.
Otherwise $U$ is independent and contains no special set, so
$|N_G(U)|\le\alpha_\ell$.
If $a=1$, then $j-1=\ell\ge2$, and
\eqref{eq:special-extension-bound} gives
$|\mathcal S_j(U)|=|\Sigma_j(U)|
\le C\frac{\alpha_{j-1}}{\alpha_j}
=C\Delta^{\frac{1}{m-1}}$.
If $a\ge2$, then $a=j-\ell\le m-3$, so
\eqref{eq:special-incidence-bound} yields
\[
|\mathcal S_j(U)|
\le C\alpha_\ell\Delta^{a-1}
\frac{\alpha_{j-1}}{\alpha_j^2}
=
C\Delta^{a-1+\frac{a+1}{m-1}}
\le C\Delta^{a-\frac{1}{m-1}}.
\]
Consequently
\begin{equation}\label{eq:special-fixed-set}
|\mathcal S_j(U)|
\le
\begin{cases}
C\Delta^{\frac{1}{m-1}},&a=1,\\[1mm]
C\Delta^{a-\frac{1}{m-1}},&a\ge2.
\end{cases}
\end{equation}

Let $q$ be a positive integer. Define a bipartite graph
$\mathcal A_q$ with vertex classes
$P=\{p_v:v\in V(G)\}$ and
$Q=\{z_{v,c}:v\in V(G),\ c\in[q]\}$, and call
$e_{v,c}=\{p_v,z_{v,c}\}$ the assignment edge for $v$ and $c$.
Thus $\mathcal A_q$ is the disjoint union of a $q$-edge star centered at
$p_v$ for every $v\in V(G)$. A $P$-perfect matching selects exactly one
edge $e_{v,\varphi(v)}$ at each star and therefore determines a coloring
$\varphi:V(G)\to[q]$.

Define a configuration hypergraph $\mathcal C_q$ on $E(\mathcal A_q)$ as
follows.
\begin{enumerate}
\item For every $uv\in E(G)$ and $c\in[q]$, include
$\{e_{u,c},e_{v,c}\}$.
\item For every special set $S$ and $c\in[q]$, include
$\{e_{u,c}:u\in S\}$.
\item For every $v\in V(G)$, $c\in[q]$, and
$S\in\binom{N_G(v)}m$, include $\{e_{u,c}:u\in S\}$ if $S$ is independent
and contains no special set.
\item For every non-star $J\in\mathcal B_m$, every embedding
$\phi:J\to G$, and every ordered pair of distinct colors $c,d\in[q]$,
include
\[
\{e_{\phi(x),c}:x\in X_J\}
\cup
\{e_{\phi(y),d}:y\in Y_J\},
\]
provided that both $\phi(X_J)$ and $\phi(Y_J)$ are independent and contain
no special set.
\end{enumerate}
Every configuration is a matching of $\mathcal A_q$.

Let $M$ be a $\mathcal C_q$-avoiding $P$-perfect matching, and let
$\varphi$ be the coloring determined by $M$. Configurations of type~1 imply
that $\varphi$ is proper. Suppose that $\varphi$ contains a bichromatic copy
of some $J\in\mathcal B_m$. Since $J$ is connected and $\varphi$ is proper,
the two colors are constant on the two bipartition classes and are
distinct. If $J=K_{1,m}$, then its leaf class is independent. A special
subset of the leaf class gives a configuration of type~2; if there is no
such subset, the $m$ leaves give a configuration of type~3. If $J$ is not
a star, then both bipartition classes of the copy are independent. If
either contains a special set, type~2 occurs; otherwise the whole copy gives
a configuration of type~4. Each possibility contradicts
$\mathcal C_q$-avoidance. Hence a $\mathcal C_q$-avoiding $P$-perfect
matching gives the required coloring.

We next count the embeddings that define configurations of type~4.
Let $J\in\mathcal B_m$ be a non-star with $i$ vertices, let $U_J$
be an $\ell$-set contained in one bipartition class of $J$, where
$1\le\ell<i$, and prescribe distinct images for the vertices of $U_J$.
We count only embeddings $\phi:J\to G$ extending this prescription
for which both $\phi(X_J)$ and $\phi(Y_J)$ are independent and contain
no special set. Put $a=i-\ell$. Since
$J$ is connected, the vertices of $V(J)\setminus U_J$ can be ordered as
$w_1,\ldots,w_a$ so that every $w_j$ has a neighbor in
$U_J\cup\{w_1,\ldots,w_{j-1}\}$. Define
$b_j = \left|N_J(w_j)\cap \bigl(U_J\cup\{w_1,\ldots,w_{j-1}\}\bigr)\right|$.
Every edge of $J$ is counted exactly once by its later endpoint, because
$U_J$ lies in one bipartition class. Consequently
$b_j\ge1$, $\sum_{j=1}^ab_j=m$, and
$\sum_{j=1}^a(b_j-1)=m-a$.

If $b_j=1$, the image of $w_j$ has at most $\Delta$ choices. If $b_j\ge2$,
its already embedded neighbors lie in one bipartition class of the copy,
and their images form an independent set containing no special set. Since $J$ is not a star,
$b_j\le m-1$, and the number of common neighbors of those images is at
most $\alpha_{b_j}$. If $h=|\{j:b_j\ge2\}|$, the number of extensions is
therefore at most
\begin{equation}\label{eq:embedding-avoiding}
\Delta^{a-h}
\prod_{b_j\ge2}\alpha_{b_j}
=
\Delta^{\frac{m(a-1)}{m-1}-\varepsilon h}.
\end{equation}
If $\ell=1$ and $J$ contains a cycle, then $m-a=m-i+1\ge1$, so $h\ge1$
and the number of extensions is at most
$\Delta^{\frac{m(i-2)}{m-1}-\varepsilon}$. If $\ell\ge2$, then
$m-a=m-i+\ell\ge\ell-1\ge1$, and hence the number is at most
$\Delta^{\frac{m(i-\ell-1)}{m-1}-\varepsilon}$. When $i-\ell=1$ and all
prescribed positions lie in one bipartition class, $J$ would be a star, so
this case does not occur. Finally, if $\ell=1$ and $J$ is a tree, then
$i=m+1$; rooting at the prescribed vertex gives the simpler upper bound
$\Delta^m$.

Fix an assignment edge $e_{u,c}$. We estimate the number of
$i$-configurations containing $e_{u,c}$ for each possible range of $i$.

If $i=2$, only types~1 and~2 occur. Type~1 contributes
$d_G(u)\le\Delta$. Since $\frac{1}{m-1}+\varepsilon<1$,
\eqref{eq:special-extension-bound} bounds the type~2 contribution by
$C\frac{\alpha_1}{\alpha_2}
=C\Delta^{\frac{1}{m-1}+\varepsilon}\le C\Delta$.

If $3\le i\le m-1$, only types~2 and~4 occur.
By \eqref{eq:special-one-vertex}, type~2 contributes at most
$C\Delta^{i-1+\varepsilon}$.
For type~4, fix $J\in\mathcal B_m$ with $|V(J)|=i$ and a vertex
of $J$ whose image is $u$. There are $O(1)$ choices for these data.
Since $J$ is connected and $i\le m$, it contains a cycle.
The color on the bipartition class containing the prescribed vertex
is $c$, and the other color has at most $q$ choices.
By \eqref{eq:embedding-avoiding}, the number of these configurations
is $O(q\Delta^{\frac{m(i-2)}{m-1}-\varepsilon})$.

If $i=m$, only types~3 and~4 occur. For type~3, its center is a
vertex of $N_G(u)$, and its other $m-1$ vertices lie in the
neighborhood of that center. Thus there are at most
$\Delta\cdot\Delta^{m-1}=\Delta^m$ choices.
For type~4, applying \eqref{eq:embedding-avoiding} with $i=m$ and
$\ell=1$, and allowing at most $q$ choices for the second color, gives at most
$O(q\Delta^{\frac{m(m-2)}{m-1}-\varepsilon})$ further configurations.

If $i=m+1$, only type~4 occurs, and the corresponding graph $J$
is a tree. After fixing $J$ and the position whose image is $u$,
root $J$ at that position and embed each remaining vertex after
its parent. Each image has at most $\Delta$ choices, giving at
most $\Delta^m$ embeddings. There are at most $q$ choices for the
second color, so the contribution is $O(q\Delta^m)$.

Taking a sufficiently large constant $A=A(m,F)$ gives
\begin{itemize}
    \item \refstepcounter{equation}\label{eq:avoiding-degree-two}
    $\Delta_2(\mathcal C_q)\le A\Delta$;\nobreak\hfill\mbox{\normalfont(\theequation)}
    \item \refstepcounter{equation}\label{eq:avoiding-degree-middle}
    $\Delta_i(\mathcal C_q)
    \le A\left(\Delta^{i-1+\varepsilon}
    +q\Delta^{\frac{m(i-2)}{m-1}-\varepsilon}\right)$
    for $3\le i\le m-1$;\nobreak\hfill\mbox{\normalfont(\theequation)}
    \item \refstepcounter{equation}\label{eq:avoiding-degree-m}
    $\Delta_m(\mathcal C_q)
    \le A\left(\Delta^m
    +q\Delta^{\frac{m(m-2)}{m-1}-\varepsilon}\right)$;\nobreak\hfill\mbox{\normalfont(\theequation)}
    \item \refstepcounter{equation}\label{eq:avoiding-degree-top}
    $\Delta_{m+1}(\mathcal C_q)\le Aq\Delta^m$.\nobreak\hfill\mbox{\normalfont(\theequation)}
\end{itemize}
When $m=3$, the range $3\le i\le m-1$ is empty.
Also, no type~4 configuration has size three, since a non-star
connected bipartite graph has at least four vertices.

We next estimate $\Delta_{i,\ell}(\mathcal C_q)$.
Fix $2\le\ell<i\le m+1$, an $\ell$-set $\mathcal S$ of assignment
edges, and put $a=i-\ell$.
If $\mathcal S$ is not a matching in $\mathcal A_q$, it is contained
in no configuration. Otherwise write
$\mathcal S=\{e_{u_1,c_1},\ldots,e_{u_\ell,c_\ell}\}$, where
$u_1,\ldots,u_\ell$ are distinct, and put
$U=\{u_1,\ldots,u_\ell\}$.
Since $i\ge3$, type~1 contributes nothing.

Type~2 can occur only when $i\le m-1$ and
$c_1=\cdots=c_\ell$.
Its contribution is $|\mathcal S_i(U)|$.
By \eqref{eq:special-fixed-set}, this is at most
$C\Delta^{\frac{1}{m-1}}$ if $a=1$ and at most
$C\Delta^{a-\frac{1}{m-1}}$ if $a\ge2$.
In either case it is at most $C\Delta^a$.

Type~3 can occur only when $i=m$, all prescribed colors agree,
and $U$ is independent and contains no special set.
Since $2\le\ell\le m-1$, the possible centers form the set
$N_G(U)$ of size at most $\alpha_\ell$.
For a fixed center, the remaining $a=m-\ell$ vertices have at most
$\Delta^a$ choices in its neighborhood. Hence the contribution
is at most
$\alpha_\ell\Delta^a
=\Delta^{\frac{m-\ell}{m-1}-\varepsilon+a}
=\Delta^{\frac{ma}{m-1}-\varepsilon}$.

For type~4, fix a non-star $J\in\mathcal B_m$ with $i$ vertices
and distinct positions $v_1,\ldots,v_\ell\in V(J)$ with prescribed
images $\phi(v_s)=u_s$ and colors $c_s$.
There are $O(1)$ choices for these data.
Put $U_J=\{v_1,\ldots,v_\ell\}$.
If the prescribed colors are not constant on each of
$U_J\cap X_J$ and $U_J\cap Y_J$, or if both intersections are nonempty
and their colors agree, there is no such configuration.

Suppose that $U_J$ meets both bipartition classes and that the
prescription is compatible. Both colors are then fixed.
Starting with $W_0=U_J$, for $1\le s\le a$ choose an edge
$t_sw_s\in E(J)$ with $t_s\in W_{s-1}$ and
$w_s\notin W_{s-1}$, and put $W_s=W_{s-1}\cup\{w_s\}$.
Such an edge exists because $J$ is connected and $W_{s-1}$ is a
nonempty proper subset of $V(J)$.
When $\phi(w_s)$ is chosen, the image $\phi(t_s)$ is already fixed,
and $\phi(w_s)\in N_G(\phi(t_s))$ has at most $\Delta$ choices.
All other required adjacencies, injectivity, and restrictions on
special sets can only reduce this number.
Thus there are at most $\Delta^a$ extensions of the prescribed images.

Suppose instead that $U_J$ lies in one bipartition class.
The other color has at most $q$ choices.
Since $J$ is connected and is not a star, each bipartition class
has at least two vertices, so $a\ge2$.
The estimate \eqref{eq:embedding-avoiding}, with $\ell\ge2$, gives
at most $q\Delta^{\frac{m(a-1)}{m-1}-\varepsilon}$ extensions,
including the choice of the other color.
Combining the contributions from configuration types~2, 3, and~4,
and increasing $A$ if necessary gives
\begin{equation}\label{eq:avoiding-codegrees}
\Delta_{i,\ell}(\mathcal C_q)
\le
A\left(
\Delta^a
+
\begin{cases}
\Delta^{\frac{ma}{m-1}-\varepsilon},&i=m,\\
0,&i\ne m,
\end{cases}
+
\begin{cases}
q\Delta^{\frac{m(a-1)}{m-1}-\varepsilon},&a\ge2,\\
0,&a=1.
\end{cases}
\right).
\end{equation}

To estimate the parameters involving $2$-configurations, let $R$
be the simple graph on $V(G)$ with
$E(R)=E(G)\cup\{S:S\text{ is a special pair}\}$.
By \eqref{eq:special-extension-bound}, for every $v\in V(G)$,
$d_R(v)\le d_G(v)+C\frac{\alpha_1}{\alpha_2}
\le\Delta+C\Delta^{\frac{1}{m-1}+\varepsilon}$.
Since $\frac{1}{m-1}+\varepsilon<1$, it follows that
$\Delta(R)=O(\Delta)$.

For $e\in E(\mathcal A_q)$, write
$\mathcal N_2(e)=\{f\in E(\mathcal A_q):
\{e,f\}\in E_2(\mathcal C_q)\}$.
Only types~1 and~2 give $2$-configurations.
Their definitions therefore give
$\mathcal N_2(e_{v,c})=\{e_{w,c}:w\in N_R(v)\}$.

We first estimate $\codeg_2(\mathcal A_q,\mathcal C_q)$.
Fix $x\in V(\mathcal A_q)$ and $e_{v,c}\in E(\mathcal A_q)$
with $x\notin e_{v,c}$.
A $2$-configuration containing $e_{v,c}$ is uniquely determined
by its other assignment edge, so
\[
\codeg_2(\mathcal A_q,\mathcal C_q;x,e_{v,c})
=|\{f\in\mathcal N_2(e_{v,c}):x\in f\}|.
\]
If $x=p_u$, the assignment edges containing $x$ are exactly
$\{e_{u,d}:d\in[q]\}$.
Their intersection with $\mathcal N_2(e_{v,c})$ is
$\{e_{u,c}\}$ when $u\in N_R(v)$, and is empty otherwise.
If $x=z_{u,d}$, the only assignment edge containing $x$ is
$e_{u,d}$, which belongs to $\mathcal N_2(e_{v,c})$ precisely when
$d=c$ and $u\in N_R(v)$.
Thus the defining set has size at most one in both cases.
Hence
$\codeg_2(\mathcal A_q,\mathcal C_q)\le1$.

We next estimate $\cdeg_2(\mathcal C_q)$.
Fix distinct assignment edges $e_{u,c}$ and $e_{v,d}$.
Their common $2$-neighbors form the set
\[
\mathcal N_2(e_{u,c})\cap\mathcal N_2(e_{v,d})
=
\begin{cases}
\{e_{w,c}:w\in N_R(u)\cap N_R(v)\},&c=d,\\
\emptyset,&c\ne d.
\end{cases}
\]
If $c=d$, distinctness of the assignment edges implies $u\ne v$,
and this set has size
$|N_R(u)\cap N_R(v)|
\le\min\{d_R(u),d_R(v)\}\le\Delta(R)$.
If $c\ne d$, the set is empty. Thus
\begin{equation}\label{eq:avoiding-cdeg2}
\cdeg_2(\mathcal C_q)\le\Delta(R)=O(\Delta).
\end{equation}

We apply \textbf{Lemma~\ref{lem:DP}} with
$\theta=\frac{\varepsilon(m-1)}{2m}$. Here $\mathcal A_q$ is $2$-bounded and $\mathcal C_q$ is
$(m+1)$-bounded. Let $D_\theta$ and $\xi>0$ be the corresponding
constants. Choose
$K=K(m,F,\xi)$ sufficiently large and put
\[
D
=
\left\lceil
K\left(\frac{\Delta^m}{\log\Delta}\right)^{\frac{1}{m-1}}
\right\rceil,
\qquad
q=2D.
\]
For sufficiently large $\Delta$, we have $D\ge\max\{D_\theta,1\}$.
Every vertex of $P$ has degree $2D\ge(1+D^{-\xi})D$, and every
vertex of $Q$ has degree $1\le D$.
Since $\mathcal A_q$ is a simple graph,
$\codeg(\mathcal A_q)\le1\le D^{1-\theta}$.
Thus the required degree and pair-codegree conditions hold.

We verify the bounds on $\Delta_i(\mathcal C_q)$ using
\eqref{eq:avoiding-degree-two}--\eqref{eq:avoiding-degree-top}.
For fixed $K$, the definition of $D$ gives
$D=(K+o(1))\Delta^{\frac{m}{m-1}}
(\log\Delta)^{-\frac{1}{m-1}}$ and
$\log D=(\frac{m}{m-1}+o(1))\log\Delta$.
In particular, $\frac{\Delta}{D\log D}\to0$, so
\eqref{eq:avoiding-degree-two} gives the required bound for $i=2$.
For the first term in \eqref{eq:avoiding-degree-middle}, since
$\varepsilon<\frac{i-1}{m-1}$ for $3\le i\le m-1$,
\begin{equation}\label{eq:avoiding-degree-ratio-middle}
\frac{\Delta^{i-1+\varepsilon}}{D^{i-1}\log D}
=
O\!\left(
\Delta^{\varepsilon-\frac{i-1}{m-1}}
(\log\Delta)^{\frac{i-1}{m-1}-1}
\right)
=o(1).
\end{equation}
For the terms involving $q$ in \eqref{eq:avoiding-degree-middle} and
\eqref{eq:avoiding-degree-m}, using $q=2D$ gives, for $3\le i\le m$,
\begin{equation}\label{eq:avoiding-degree-ratio-q}
\frac{q\Delta^{\frac{m(i-2)}{m-1}-\varepsilon}}
{D^{i-1}\log D}
=
O\!\left(
\Delta^{-\varepsilon}
(\log\Delta)^{\frac{i-2}{m-1}-1}
\right)
=o(1).
\end{equation}
The remaining terms in \eqref{eq:avoiding-degree-m} and
\eqref{eq:avoiding-degree-top} satisfy
\begin{equation}\label{eq:avoiding-degree-ratio-top}
\frac{\Delta^m}{D^{m-1}\log D}
\longrightarrow\frac{m-1}{mK^{m-1}},
\qquad
\frac{q\Delta^m}{D^m\log D}
\longrightarrow\frac{2(m-1)}{mK^{m-1}}.
\end{equation}
Choose $K$ so large that
$\frac{2A(m-1)}{mK^{m-1}}<\frac{\xi}{2}$.
For $i=2$, \eqref{eq:avoiding-degree-two} and
$\frac{\Delta}{D\log D}\to0$ apply; for $3\le i\le m+1$,
\eqref{eq:avoiding-degree-middle}--\eqref{eq:avoiding-degree-top} and
\eqref{eq:avoiding-degree-ratio-middle}, \eqref{eq:avoiding-degree-ratio-q}, and
\eqref{eq:avoiding-degree-ratio-top}
give the required bounds. Thus
$\Delta_i(\mathcal C_q)\le\xi D^{i-1}\log D$
for every $2\le i\le m+1$, once $\Delta$ is sufficiently large.

Now fix $2\le\ell<i\le m+1$ and put $a=i-\ell$.
Since $a\ge1>m\theta$ and
$\varepsilon-\frac{m\theta}{m-1}=\frac{\varepsilon}{2}$,
the terms in \eqref{eq:avoiding-codegrees} satisfy
\begin{itemize}
    \item \refstepcounter{equation}\label{eq:avoiding-codegree-ratio-ordinary}
    $\frac{\Delta^a}{D^{a-\theta}}
    =
    O\!\left(
    \Delta^{-\frac{a-m\theta}{m-1}}
    (\log\Delta)^{\frac{a-\theta}{m-1}}
    \right)
    =o(1)$;\nobreak\hfill\mbox{\normalfont(\theequation)}
    \item \refstepcounter{equation}\label{eq:avoiding-codegree-ratio-star}
    $\frac{\Delta^{\frac{ma}{m-1}-\varepsilon}}{D^{a-\theta}}
    =
    O\!\left(
    \Delta^{-\frac{\varepsilon}{2}}
    (\log\Delta)^{\frac{a-\theta}{m-1}}
    \right)
    =o(1)$;\nobreak\hfill\mbox{\normalfont(\theequation)}
    \item \refstepcounter{equation}\label{eq:avoiding-codegree-ratio-one-side}
    $\frac{q\Delta^{\frac{m(a-1)}{m-1}-\varepsilon}}{D^{a-\theta}}
    =
    O\!\left(
    \Delta^{-\frac{\varepsilon}{2}}
    (\log\Delta)^{\frac{a-1-\theta}{m-1}}
    \right)
    =o(1)$ for $a\ge2$.\nobreak\hfill\mbox{\normalfont(\theequation)}
\end{itemize}
By \eqref{eq:avoiding-codegrees} and
\eqref{eq:avoiding-codegree-ratio-ordinary},
\eqref{eq:avoiding-codegree-ratio-star}, and
\eqref{eq:avoiding-codegree-ratio-one-side},
$\Delta_{i,\ell}(\mathcal C_q)\le D^{i-\ell-\theta}$
for every $2\le\ell<i\le m+1$ and sufficiently large $\Delta$.

Finally, $\codeg_2(\mathcal A_q,\mathcal C_q)\le1\le D^{1-\theta}$.
Taking $a=1$ in \eqref{eq:avoiding-codegree-ratio-ordinary} gives
$\Delta=o(D^{1-\theta})$, so
\eqref{eq:avoiding-cdeg2} implies
$\cdeg_2(\mathcal C_q)\le D^{1-\theta}$ for sufficiently large $\Delta$.
All hypotheses of \textbf{Lemma~\ref{lem:DP}} are now satisfied.
It gives a $\mathcal C_q$-avoiding $P$-perfect matching, and hence
the required coloring with
$q=O((\frac{\Delta^m}{\log\Delta})^{\frac{1}{m-1}})$ colors.

\subsection{Proof of Theorem~\ref{thm:B-k}}\label{ss:upper-B}

We first prove the middle case, where $s\ge3$ and $k\le1$.
Let $\Gamma$ be the graph with vertex set $E(G)$ in which distinct
$e,f\in E(G)$ are adjacent if $e\cap f\ne\emptyset$ or
$f\in\Omega(e)$.
A proper vertex-coloring of $\Gamma$ is exactly a B-coloring of $G$.
For every $e\in E(G)$, \textbf{Lemma~\ref{lem:B-local-k}} gives
$d_\Gamma(e)\le2(\Delta-1)+|\Omega(e)|
\le(4h-2)(\Delta-1)$.
A greedy coloring of $\Gamma$ therefore gives
$q_B(G)\le(4h-2)(\Delta-1)+1$, as required.

For the remaining two cases, we use a common auxiliary construction.
For a positive integer $q$, put
\[
P=\{p_e:e\in E(G)\},
\qquad
Q=\{z_{v,c}:v\in V(G),\ c\in[q]\},
\]
where the displayed symbols denote distinct vertices.
Define the hypergraph $\mathcal A_q$ by
$V(\mathcal A_q)=P\cup Q$, $P\cap Q=\emptyset$, and
$E(\mathcal A_q)=\{a_{e,c}:e\in E(G),\ c\in[q]\}$, where
$a_{e,c}=\{p_e,z_{u,c},z_{v,c}\}$ for $e=uv\in E(G)$.
We call $a_{e,c}$ the assignment edge for $e,c$.
Each edge of $\mathcal A_q$ contains one vertex of $P$ and two vertices
of $Q$. Thus $\mathcal A_q$ is a bipartite $3$-uniform hypergraph
with parts $P,Q$.

A $P$-perfect matching $M$ contains, for each $e\in E(G)$, exactly
one edge $a_{e,c}$; assign color $c$ to $e$.
If distinct edges $e=uv$ and $f=vw$ of $G$ received the same color
$c$, then $a_{e,c},a_{f,c}\in M$ would satisfy
$a_{e,c}\cap a_{f,c}=\{z_{v,c}\}\ne\emptyset$, contradicting that
$M$ is a matching. Thus the resulting edge-coloring of $G$ is proper.

Define a configuration hypergraph $\mathcal C_q$ with vertex set
$E(\mathcal A_q)$ and edge set
$\{\{a_{e,c},a_{f,c}\}:e,f\in E(G),\ f\in\Omega(e),\ c\in[q]\}$.
Its edges have size two.
If $f\in\Omega(e)$, then $e\cap f=\emptyset$, so
$a_{e,c}\cap a_{f,c}=\emptyset$.
Hence every configuration is a matching in $\mathcal A_q$.
For a $P$-perfect matching $M$, avoiding these configurations means
that no two opposite edges of a $4$-cycle in $G$ receive the same color.
Together with properness, this is precisely the condition for a
B-coloring. Conversely, a B-coloring $\varphi:E(G)\to[q]$ gives
the $\mathcal C_q$-avoiding $P$-perfect matching
$\{a_{e,\varphi(e)}:e\in E(G)\}$.

No edge of $\mathcal A_q$ contains two vertices of $P$.
A pair $\{p_e,z_{v,c}\}$ is contained only in $a_{e,c}$ when
$v\in e$, and in no edge otherwise.
Two distinct vertices $z_{u,c},z_{v,d}\in Q$ lie in an edge
exactly when $c=d$ and $uv\in E(G)$, in which case that edge is
$a_{uv,c}$. Consequently $\codeg(\mathcal A_q)\le1$.
Also, $d_{\mathcal A_q}(p_e)=q$ and
$d_{\mathcal A_q}(z_{v,c})=d_G(v)\le\Delta$. Moreover,
\begin{equation}\label{eq:B-configuration-degree}
\Delta_2(\mathcal C_q)=\max\limits_{e\in E(G)}|\Omega(e)|.
\end{equation}

We next estimate $\codeg_2(\mathcal A_q,\mathcal C_q)$.
Fix $x\in V(\mathcal A_q)$ and $a_{e,c}\in E(\mathcal A_q)$
with $x\notin a_{e,c}$.
If $x=p_f$, the only edge of $\mathcal A_q$ containing $x$
that can form a configuration with $a_{e,c}$ is $a_{f,c}$,
and it does so exactly when $f\in\Omega(e)$.
If $x=z_{v,d}$, such an edge must have color $d=c$ and must
be $a_{f,c}$ for some $f\in\Omega(e)$ with $v\in f$.
There are $\rho(e,v)$ such edges when $d=c$, and none when $d\ne c$.
Thus
\begin{equation}\label{eq:B-assignment-codegree}
\codeg_2(\mathcal A_q,\mathcal C_q)
\le\max\left\{1,\max\limits_{e\in E(G),\,v\in V(G)}\rho(e,v)\right\}.
\end{equation}

For distinct assignment edges $a_{e,c}$ and $a_{f,d}$, a common
$2$-neighbor must have both colors $c$ and $d$, so none exists
when $c\ne d$. If $c=d$, then $e\ne f$, and their common
$2$-neighbors are exactly
$\{a_{g,c}:g\in\Omega(e)\cap\Omega(f)\}$.
Thus
\begin{equation}\label{eq:B-assignment-common-degree}
\cdeg_2(\mathcal C_q)
\le\max\bigl(\{0\}\cup
\{|\Omega(e)\cap\Omega(f)|:e,f\in E(G),\ e\ne f\}\bigr).
\end{equation}

Suppose first that $s\le2$, and put
$t=\max\{2,|X_F|,|Y_F|\}$.
Since $F\subseteq K_{2,t}$, the graph $G$ is $K_{2,t}$-free.
By \textbf{Lemma~\ref{lem:B-local-k}},
$|\Omega(e)|\le(t-1)\Delta$ and $\rho(e,v)\le2(t-1)$.
Since $t\ge2$, \textbf{Lemma~\ref{lem:B-local-k}} and
\eqref{eq:B-configuration-degree}, \eqref{eq:B-assignment-codegree}, and
\eqref{eq:B-assignment-common-degree} yield
\[
\Delta_2(\mathcal C_q)\le(t-1)\Delta,\qquad
\codeg_2(\mathcal A_q,\mathcal C_q)\le2(t-1),\qquad
\cdeg_2(\mathcal C_q)\le4(t-1)^2.
\]
We apply \textbf{Lemma~\ref{lem:DP}} with $r=3$, $g=2$, and
$\theta=\frac{1}{2}$.
Since $\mathcal C_q$ is $2$-uniform, the configuration-degree
condition only needs to be checked for $i=2$, and there are no
integers satisfying $2\le\ell<i\le2$ for the higher
configuration-codegree conditions.
Let $D_{\frac{1}{2}}$ and $\xi>0$ be the constants supplied by the lemma.
Set $\eta=\min\{\xi,\frac{1}{2}\}$, $D=\Delta$, and
$q=\lceil\Delta+\Delta^{1-\eta}\rceil$.
Every vertex of $P$ has degree $q\ge(1+D^{-\xi})D$, and every
vertex of $Q$ has degree at most $D$.
For sufficiently large $\Delta$, we have $D\ge D_{\frac{1}{2}}$,
$\Delta_2(\mathcal C_q)\le\xi D\log D$,
$\codeg_2(\mathcal A_q,\mathcal C_q)\le D^{\frac{1}{2}}$,
$\cdeg_2(\mathcal C_q)\le D^{\frac{1}{2}}$, and
$\codeg(\mathcal A_q)\le1\le D^{\frac{1}{2}}$.
Thus \textbf{Lemma~\ref{lem:DP}} gives a
$\mathcal C_q$-avoiding $P$-perfect matching, and hence
$q_B(G)\le q\le\Delta+\Delta^{1-\eta}+1$.

It remains to consider $k\ge2$. Put
\[
\lambda_k=2-\frac{1}{k},
\qquad
\mu_k=2-\frac{1}{k-1}.
\]
Then $1\le\mu_k<\lambda_k$.
Choose $\theta=\frac{\lambda_k-\mu_k}{2\lambda_k}\in(0,1)$.
By \textbf{Lemma~\ref{lem:B-local-k}},
$\rho(e,v)\le2(\Delta-1)$.
For $\Delta\ge2$, \textbf{Lemma~\ref{lem:B-local-k}} and
\eqref{eq:B-assignment-codegree} and \eqref{eq:B-assignment-common-degree}
therefore give
\begin{equation}\label{eq:B-parameter-bounds}
\codeg_2(\mathcal A_q,\mathcal C_q)\le2(\Delta-1),\qquad
\cdeg_2(\mathcal C_q)\le C\Delta^{\mu_k},
\end{equation}
where $C$ depends only on $F$.

Apply \textbf{Lemma~\ref{lem:DP}} with $r=3$, $g=2$, and this
value of $\theta$, and let $D_\theta$ and $\xi>0$ be its constants.
Choose $K=K(F,\xi)$ sufficiently large and put
\[
L=\frac{\Delta^{\lambda_k}}{\log\Delta},
\qquad
D=KL,
\qquad
q=\left\lceil(1+D^{-\xi})D\right\rceil.
\]
Since $\lambda_k>1$, we have $\frac{D}{\Delta}\to\infty$.
Thus $D\ge D_\theta$ and every vertex of $Q$ has degree at most $D$
for sufficiently large $\Delta$.
Every vertex of $P$ has degree $q\ge(1+D^{-\xi})D$.
By \textbf{Lemma~\ref{lem:B-local-k}},
$\Delta_2(\mathcal C_q)\le C\Delta^{\lambda_k}$.
Since $\log D=\Theta(\log\Delta)$, choosing $K$ sufficiently
large gives $\Delta_2(\mathcal C_q)\le\xi D\log D$.
The hypergraph $\mathcal C_q$ is $2$-uniform, so this is the only
configuration-degree condition, and no higher
configuration-codegree condition is required.
Also,
$\lambda_k(1-\theta)=\frac{\lambda_k+\mu_k}{2}>\mu_k\ge1$,
so \eqref{eq:B-parameter-bounds} gives
\[
\codeg_2(\mathcal A_q,\mathcal C_q)\le D^{1-\theta},
\qquad
\cdeg_2(\mathcal C_q)\le D^{1-\theta}
\]
for sufficiently large $\Delta$.
Finally, $\codeg(\mathcal A_q)\le1\le D^{1-\theta}$.
All hypotheses of \textbf{Lemma~\ref{lem:DP}} therefore hold.

The resulting $\mathcal C_q$-avoiding $P$-perfect matching gives a
B-coloring with at most $q$ colors. Since $D=KL$ and
$D^{1-\xi}=o(L)$, for sufficiently large $\Delta$ we have
\[
q
\le
D+D^{1-\xi}+1
\le
(K+2)L.
\]
Absorbing the fixed constants into $C$ yields
$q_B(G)\le C\frac{\Delta^{2-\frac{1}{k}}}{\log\Delta}$, as required.

\section{Lower Bounds}\label{s:lower}

For a positive integer $n$ and $p\in[0,1]$, let $\mathbb G(n,p)$
denote the binomial random graph on vertex set $[n]$, in which each
pair of distinct vertices is an edge independently with probability $p$.
For positive integers $a,b$, let $\mathbb G_{a,b}(p)$ denote the
binomial random bipartite graph with two fixed vertex classes of sizes
$a,b$, in which each pair with one vertex in each class is an edge
independently with probability $p$.

\subsection{Proof of Theorem~\ref{thm:lower-avoiding}}\label{ss:lower-avoiding}

\begin{lemma}\label{lem:random-tree-free}
Let $H$ be a fixed tree with $m\ge2$ edges.
There exist an integer $a_H\ge1$ and constants $\rho_H,C>0$ such that,
for every integer $a\ge a_H$ and every $p\in[0,1]$ with $ap\le\rho_H$,
\[
\mathbb P(\mathbb G_{a,a}(p)\text{ is }H\text{-free})
\le\exp(-Ca^{m+1}p^m).
\]
\end{lemma}

\begin{proof}
Put $h=m+1$ and $a_H=2h$, and fix a bipartition
$V(H)=X_H\cup Y_H$ with $X_H\cap Y_H=\emptyset$.
Let $A_0,B_0$ be disjoint sets of size $a$, and let $K$ be the
complete bipartite graph with these vertex classes.
Write $\mathcal T$ for the set of distinct subgraphs of $K$ that are
images of injective maps $\phi:V(H)\to A_0\cup B_0$ satisfying
$\phi(X_H)\subseteq A_0$ and $\phi(Y_H)\subseteq B_0$.
Each subgraph belongs to $\mathcal T$ only once, regardless of the
number of maps producing it.
Let $X\sim\mathbb G_{a,a}(p)$ on these classes, and put
$Z=|\{J\in\mathcal T:E(J)\subseteq E(X)\}|$.

If $p=0$, the asserted probability bound is immediate, so assume $p>0$.
There are at most $a^h$ such injective maps.
Choosing their images successively gives at least
$(a-h+1)^h\ge(\frac{a}{2})^h$ maps, since $a\ge2h$.
Each member of $\mathcal T$ is produced by at most $h!$ maps.
Thus, with $b_H=\frac{1}{2^h h!}$,
\[
b_Ha^hp^m\le\mu:=\mathbb EZ=|\mathcal T|p^m\le a^hp^m.
\]

For the events $\{E(J)\subseteq E(X)\}$, indexed by $J\in\mathcal T$,
the sum $\Gamma$ in \textbf{Lemma~\ref{lem:janson}} is
\[
\Gamma=
\sum_{\substack{J_1,J_2\in\mathcal T,\ J_1\ne J_2\\
E(J_1)\cap E(J_2)\ne\emptyset}}
p^{|E(J_1)\cup E(J_2)|}.
\]
Fix such a pair, and put
$\ell=|E(J_1)\cap E(J_2)|$ and
$s=|V(J_1)\cap V(J_2)|$.
The graph with vertex set $V(J_1)\cap V(J_2)$ and edge set
$E(J_1)\cap E(J_2)$ is a subgraph of the tree $J_1$, so it is
a forest. Consequently $s\ge\ell+1$.
Also $1\le\ell\le m-1$: equality $\ell=m$ would force the two
edge sets, and hence the two copies of $H$, to coincide.

We bound the number $N_{\ell,s}$ of ordered pairs with these two
intersection sizes by counting the injective maps producing them.
Let $\phi_1,\phi_2$ produce $J_1,J_2$, respectively, and define
$S_i=\phi_i^{-1}(V(J_1)\cap V(J_2))\subseteq V(H)$ for $i=1,2$.
Both sets have size $s$, and the rule
$\tau(x)=\phi_2^{-1}(\phi_1(x))$ defines a bijection
$\tau:S_1\to S_2$.
It preserves membership in $X_H$ and $Y_H$.
There are at most $\binom{h}{s}^2s!$ choices for $S_1,S_2,\tau$.
For fixed choices, there are at most $a^h$ choices for $\phi_1$.
The identities $\phi_2(\tau(x))=\phi_1(x)$ for $x\in S_1$
then determine the images under $\phi_2$ of all vertices of $S_2$.
Each of its remaining $h-s$ vertices has at most $a$ choices in
its prescribed class. Requiring injectivity, exactly $s$ common
vertices, and exactly $\ell$ common edges can only reduce this number.
Since every pair $J_1,J_2$ is produced by at least one pair of maps,
$N_{\ell,s}\le\binom{h}{s}^2s!\,a^{2h-s}$.

Put $K_H=\sum_{s=0}^{h}\binom{h}{s}^2s!$ and $C'=\frac{K_H}{b_H}$.
Since $|E(J_1)\cup E(J_2)|=2m-\ell$ and $s\ge\ell+1$,
\[
\begin{aligned}
\Gamma
=\sum_{\ell=1}^{m-1}\sum_{s=\ell+1}^{h}
N_{\ell,s}p^{2m-\ell}
\le K_H\sum_{\ell=1}^{m-1}
a^{2m+1-\ell}p^{2m-\ell}
\le C'\mu\sum_{\ell=1}^{m-1}(ap)^{m-\ell}.
\end{aligned}
\]
Take
$\rho_H=\min\{1,\frac{1}{C'(m-1)}\}$.
If $ap\le\rho_H$, then
$C'\sum_{\ell=1}^{m-1}(ap)^{m-\ell}
\le C'(m-1)ap\le1$, and hence $\Gamma\le\mu$.
By \textbf{Lemma~\ref{lem:janson}},
$\mathbb P(Z=0)\le\exp(-\frac{\mu}{2})
\le\exp(-\frac{b_H}{2}a^{m+1}p^m)$.
Since an $H$-free graph contains none of the copies indexed by
$\mathcal T$, it follows, with $C=\frac{b_H}{2}$, that
\[
\mathbb P(\mathbb G_{a,a}(p)\text{ is }H\text{-free})
\le\mathbb P(Z=0)
\le\exp(-Ca^{m+1}p^m).
\]
\end{proof}

We now complete the proof.
Let $a_H,\rho_H,C'$ be the constants in
\textbf{Lemma~\ref{lem:random-tree-free}}, and put
$\kappa_H=\frac{C'}{4\cdot16^{m+1}}$.
First fix a constant
$0<\eta_H\le\min\{1,(\frac{\kappa_H}{4})^{\frac{1}{m-1}}\}$.
For each integer $d\ge3$, define
$q(d)=\lfloor\eta_H(\frac{d^m}{\log d})^{\frac{1}{m-1}}\rfloor$.
Since $q(d)\to\infty$ and $\frac{q(d)}{d}\to\infty$ as $d\to\infty$,
there is an integer $d_0\ge3$ such that, for every integer $d\ge d_0$,
\[
q(d)\ge2,\qquad
\frac{d}{8q(d)}\le\rho_H,\qquad
10e^{-\frac{d}{6}}\le\frac{1}{4}.
\]
Fix any such $d$ and write $q=q(d)$.
We will choose $n$ after fixing $d$ and $q$.
More precisely, put
$B_g(d)=\sum_{\ell=3}^{g-1}\frac{d^\ell}{2\ell}$,
where the sum is zero if $g=3$, and choose an integer $n$ satisfying
\[
n\ge
\max\{d+1,\ 16q,\ 8q(a_H+1),\ 40B_g(d)+1\}.
\]
Put $p=\frac{d}{n}$, let $G\sim\mathbb G(n,p)$, and set
$a=\lfloor\frac{n}{8q}\rfloor$.
The choices of $d,n$ give
$a\ge a_H$, $a\ge\frac{n}{16q}$, and
$ap\le\frac{d}{8q}\le\rho_H$.
Thus \textbf{Lemma~\ref{lem:random-tree-free}} applies with these $a,p$.

Fix $U\subseteq[n]$ with $|U|\ge\frac{n}{2}$ and a map
$\sigma:U\to[q]$.
Split each color class into sets of size $a$ and a remainder of
size less than $a$.
Fewer than $qa\le\frac{n}{8}$ vertices are discarded, leaving at least
$3q$ full sets. Choose $q$ of these sets and denote them by
$V_1,\ldots,V_q$.
Use the natural order on $[n]$ to make all these choices depend
only on $U,\sigma$, independently of the random edges of $G$.

Let $\mathcal E_\sigma$ be the event that $\sigma$ is a proper
$(2,H)$-avoiding coloring of $G[U]$.
For $i<j$, let $G_{ij}$ be the bipartite graph with vertex classes
$V_i,V_j$ and all edges of $G$ between them.
On $\mathcal E_\sigma$, if $V_i$ and $V_j$ have the same color,
then $G_{ij}$ is edgeless by properness.
If they have different colors, a copy of $H$ in $G_{ij}$ would
be bichromatic. Thus $G_{ij}$ is $H$-free in either case.
Each $G_{ij}$ has distribution $\mathbb G_{a,a}(p)$, and these
random graphs use disjoint sets of possible edges for distinct
unordered pairs $\{i,j\}$.
Their being $H$-free are therefore mutually independent events.
By \textbf{Lemma~\ref{lem:random-tree-free}}, using
$\binom{q}{2}\ge\frac{q^2}{4}$ and $a\ge\frac{n}{16q}$, we obtain
\[
\begin{aligned}
\mathbb P(\mathcal E_\sigma)
\le
\exp\left(-C'\binom{q}{2} a^{m+1}p^m\right)
\le
\exp\left(-\kappa_H n\frac{d^m}{q^{m-1}}\right).
\end{aligned}
\]
There are at most $(q+1)^n$ pairs $U,\sigma$, since each vertex
is either outside $U$ or assigned one of the $q$ colors.
As $m\ge2$, $\eta_H\le1$, and $d\ge3$, we have $q\le d^2$ and
$\log(q+1)\le3\log d$.
Also $\frac{d^m}{q^{m-1}}\ge\eta_H^{-(m-1)}\log d$ and
$\kappa_H\eta_H^{-(m-1)}\ge4$.
Consequently
\[
\begin{aligned}
\mathbb P\bigl(\exists U\subseteq[n]:
|U|\ge\frac{n}{2},\ \chi_{2,H}(G[U])\le q\bigr)
\le
\exp\left(n\log(q+1)
-\kappa_H n\frac{d^m}{q^{m-1}}\right)
\le d^{-n}<\frac{1}{4}.
\end{aligned}
\]

Let $Y$ count the vertices of $G$ of degree at least $10d$.
For every vertex $v$, its degree is binomial with mean
$\mu_v=(n-1)p\ge\frac{d}{2}$, and $10d\ge2\mu_v$.
By \textbf{Lemma~\ref{lem:chernoff}},
$\mathbb P(d_G(v)\ge10d)
\le\exp(-\frac{\mu_v}{3})\le e^{-\frac{d}{6}}$.
Hence $\mathbb EY\le ne^{-\frac{d}{6}}$, and Markov's inequality gives
$\mathbb P(Y\ge\frac{n}{10})
\le10e^{-\frac{d}{6}}\le\frac{1}{4}$.

Let $W$ count the cycles of $G$ of lengths less than $g$.
For each $\ell\ge3$, there are at most $\frac{n^\ell}{2\ell}$
possible cycles of length $\ell$, and each occurs with probability
$p^\ell$.
Therefore $\mathbb EW\le B_g(d)$ and
$\mathbb P(W\ge\frac{n}{10})
\le\frac{10B_g(d)}{n}<\frac{1}{4}$.
The sum of the three failure probabilities is less than one.
Thus there is a realization of $G$ such that
$\chi_{2,H}(G[U])>q$ for every $U\subseteq[n]$ with
$|U|\ge\frac{n}{2}$, while $Y<\frac{n}{10}$ and $W<\frac{n}{10}$.

Delete all vertices of degree at least $10d$.
Then, while a cycle of length less than $g$ remains, delete one
of its vertices.
Vertex deletion creates no new cycle, so the second stage deletes
at most $W$ vertices.
The remaining induced subgraph $G'$ has more than $\frac{4n}{5}$
vertices, maximum degree less than $10d$, and girth at least $g$.
It follows that $\chi_{2,H}(G')>q$.

Set $\Delta_0=10d_0$, fix any integer $\Delta\ge\Delta_0$, and apply
the construction of $G'$ with $d=\lfloor\frac{\Delta}{10}\rfloor$.
Then $d\ge d_0$, $d\ge\frac{\Delta}{20}$, and $10d\le\Delta$.
Let $G^*$ be the disjoint union of $G'$ and $K_{1,\Delta}$.
Since $\Delta(G')<10d\le\Delta$, we have $\Delta(G^*)=\Delta$.
Every cycle of $G^*$ lies in $G'$, so $g(G^*)\ge g$.
Restricting a $(2,H)$-avoiding coloring of $G^*$ to $V(G')$ gives
a $(2,H)$-avoiding coloring of $G'$, and hence
$\chi_{2,H}(G^*)\ge\chi_{2,H}(G')$.
Since $q\ge2$, the definition of $q$ gives
$q\ge\frac{\eta_H}{2}(\frac{d^m}{\log d})^{\frac{1}{m-1}}$.
Using $d\ge\frac{\Delta}{20}$ and $\log d\le\log\Delta$, we obtain
\[
\chi_{2,H}(G^*)
\ge\chi_{2,H}(G')
>q
\ge
\frac{\eta_H}{2\cdot20^{\frac{m}{m-1}}}
\left(\frac{\Delta^m}{\log\Delta}\right)^{\frac{1}{m-1}}.
\]
This proves the theorem with
$C=\frac{\eta_H}{2\cdot20^{\frac{m}{m-1}}}$.

\subsection{Proof of Theorem~\ref{thm:B-lower}}\label{ss:lower-B}

Let $\alpha_B(G)$ be the maximum size of a matching $M\subseteq E(G)$ such
that no two edges of $M$ are opposite edges of a $4$-cycle. Every color
class of a B-coloring has this property. If $E(G)\ne\emptyset$, then
$q_B(G)\ge \frac{|E(G)|}{\alpha_B(G)}$.
\begin{lemma}\label{lem:random-alphaB}
There exists an absolute constant $C>0$ such that the following holds.
Let $p=p(n)\in(0,1)$ satisfy $p^{-2}\log n=o(n)$, and let
$G\sim\mathbb G_{n,n}(p)$. Then, with probability tending to one,
$\alpha_B(G)\le Cp^{-2}\log n$.
\end{lemma}

\begin{proof}
Let $X,Y$ be the two fixed vertex classes of size $n$, and let $K$
be the complete bipartite graph with classes $X,Y$.
For an integer $1\le a\le n$, let $\mathcal M_a$ be the set of matchings
of size $a$ in $K$. Choosing the vertices of a matching in each class
and a bijection between the two chosen sets gives
$|\mathcal M_a|=\binom{n}{a}^2a!$.

For $M=\{x_iy_i:1\le i\le a\}\in\mathcal M_a$, where $x_i\in X$
and $y_i\in Y$, let $I_M$ be the indicator of the event that
$M\subseteq E(G)$ and
$\{x_iy_j,x_jy_i\}\nsubseteq E(G)$ for every $1\le i<j\le a$.
This event holds exactly when $M$ is a matching in $G$ with no two
edges opposite in a $4$-cycle of $G$.
The conditions involve the pairwise disjoint sets of possible edges
$M$ and $\{x_iy_j,x_jy_i\}$, $1\le i<j\le a$.
Independence of the edge choices therefore gives
$\mathbb E I_M=p^a(1-p^2)^{\binom{a}{2}}$.
Put $Z_a=\sum_{M\in\mathcal M_a}I_M$. Then
\[
\begin{aligned}
\mathbb E Z_a
=\binom{n}{a}^2a!\,p^a(1-p^2)^{\binom{a}{2}}
\le n^{2a}\exp\left(-p^2\binom{a}{2}\right).
\end{aligned}
\]
Moreover, $\alpha_B(G)\ge a$ if and only if $Z_a\ge1$.
Since $Z_a$ takes nonnegative integer values,
\[
\begin{aligned}
\mathbb P(\alpha_B(G)\ge a)
=\sum_{r\ge1}\mathbb P(Z_a=r)
\le\sum_{r\ge1}r\,\mathbb P(Z_a=r)
=\mathbb E Z_a
\le\exp\left(2a\log n-p^2\binom{a}{2}\right).
\end{aligned}
\]

Now take $a=\lceil8p^{-2}\log n\rceil$.
The assumption $p^{-2}\log n=o(n)$ ensures $a\le n$ for all
sufficiently large $n$.
Write $a=8p^{-2}\log n+\delta$, where $0\le\delta<1$.
Substituting this expression gives
\[
\begin{aligned}
2a\log n-p^2\binom{a}{2}
&=-16p^{-2}(\log n)^2+(4-6\delta)\log n
+\frac{p^2}{2}\delta(1-\delta)\\
&\le-16p^{-2}(\log n)^2+4\log n+\frac{1}{8}
\longrightarrow-\infty,
\end{aligned}
\]
where we used $p<1$ and $\delta(1-\delta)\le\frac{1}{4}$.
Hence $\mathbb P(\alpha_B(G)\ge a)\to0$.
With probability tending to one,
$\alpha_B(G)\le a-1<8p^{-2}\log n$, proving the lemma with $C=8$.
\end{proof}

We now complete the proof.
Let $C'$ be the constant in \textbf{Lemma~\ref{lem:random-alphaB}}.
Put $s=k+1$ and
$\rho_0=\frac{s+t-1}{st}$. The hypothesis
$(k-1)(t-2)>2$ is equivalent to $(s-2)(t-2)>2$, and hence to
$\rho_0<\frac{1}{2}$. Choose a constant $\zeta=\zeta(s,t)\in(0,1)$ so small that
$16\zeta^{st}\le\frac{1}{4}$, put
$p=\zeta n^{-\rho_0}$, and let $G\sim\mathbb G_{n,n}(p)$.

Since $np=\zeta n^{1-\rho_0}\gg\log n$,
\textbf{Lemma~\ref{lem:chernoff}} and a union bound show that, with
probability tending to one,
$\Delta(G)\le2np$ and $|E(G)|\ge\frac{n^2p}{2}$.
Moreover, $\rho_0<\frac{1}{2}$ implies
$p^{-2}\log n=o(n)$, so \textbf{Lemma~\ref{lem:random-alphaB}} gives
$\alpha_B(G)\le C'p^{-2}\log n$ with probability tending to one.

Let $X$ denote the number of copies of $K_{s,t}$ in $G$. Since
$s+t-st\rho_0=1$, we have
\[
\mathbb EX
\le
2n^{s+t}p^{st}
=
2\zeta^{st}n.
\]
By Markov's inequality,
$\mathbb P(X>8\zeta^{st}n)\le\frac{1}{4}$.
Consequently, for all sufficiently large $n$, there is a realization of
$G$ for which simultaneously
\[
\Delta(G)\le2np,
\qquad
|E(G)|\ge\frac{n^2p}{2},
\qquad
\alpha_B(G)\le C'p^{-2}\log n,
\qquad
X\le8\zeta^{st}n.
\]
Fix such a realization.

Choose one vertex from each copy of $K_{s,t}$ in $G$, let $S$ be the set of
chosen vertices, and put $H=G-S$. Then $H$ is $K_{s,t}$-free
and is induced in $G$, with $|S|\le8\zeta^{st}n$.
Since deleting a vertex removes at most
$\Delta(G)$ edges,
\[
|E(H)|
\ge
\frac{n^2p}{2}-|S|\Delta(G)
\ge
\left(\frac{1}{2}-16\zeta^{st}\right)n^2p
\ge
\frac{n^2p}{4}.
\]

Since $H$ is induced in $G$, the condition defining $\alpha_B$ is
the same in $H$ and $G$ for every matching contained in $E(H)$.
Thus $\alpha_B(H)\le\alpha_B(G)\le C'p^{-2}\log n$, and hence
\[
q_B(H)
\ge
\frac{|E(H)|}{\alpha_B(H)}
\ge
\frac{1}{4C'}\frac{n^2p^3}{\log n}.
\]

Set $\Delta_n=\lceil2np\rceil$. Then
$\Delta(H)\le\Delta_n$ and $\Delta_n\to\infty$. By the definition of
$\rho_0$,
\[
\frac{2-3\rho_0}{1-\rho_0}
=
\frac{2st-3s-3t+3}{st-s-t+1}
=
2-\frac{1}{k}-\frac{k+1}{k(t-1)}
=
\gamma_{k,t}.
\]
Hence
\[
n^2p^3
=
\zeta^3n^{2-3\rho_0}
=
\zeta^{3-\gamma_{k,t}}(np)^{\gamma_{k,t}}.
\]
Since $np=\zeta n^{1-\rho_0}\to\infty$, we have
$\Delta_n\le2np+1\le3np$ for all sufficiently large $n$.
Also, since $\zeta$ is fixed and $1-\rho_0>0$, for all sufficiently
large $n$,
\[
\log\Delta_n
\ge\log(2np)
=(1-\rho_0)\log n+\log(2\zeta)
\ge\frac{1-\rho_0}{2}\log n.
\]
Thus
\begin{equation}\label{eq:B-random-lower}
q_B(H)
\ge
\frac{\zeta^{3-\gamma_{k,t}}(1-\rho_0)}
{8C'\,3^{\gamma_{k,t}}}
\frac{\Delta_n^{\gamma_{k,t}}}{\log\Delta_n}.
\end{equation}
Let $C=C(k,t)>0$ denote the coefficient in \eqref{eq:B-random-lower}.
For each sufficiently large integer $n$, the construction and
\eqref{eq:B-random-lower} give a $K_{s,t}$-free graph $H_n$ with
$\Delta(H_n)\le\Delta_n$ and
$q_B(H_n)\ge C\frac{\Delta_n^{\gamma_{k,t}}}{\log\Delta_n}$.
Since $\Delta_n=\lceil2\zeta n^{1-\rho_0}\rceil$ and $0<\rho_0<1$, we have
\[
0<2\zeta\bigl((n+1)^{1-\rho_0}-n^{1-\rho_0}\bigr)
\le2\zeta(1-\rho_0)n^{-\rho_0}\longrightarrow0.
\]
Choose an integer $n_0$ such that $H_n$ exists and
$\Delta_{n+1}-\Delta_n\in\{0,1\}$ for every $n\ge n_0$, and set
$\Delta_0=\Delta_{n_0}$.
The integer sequence $(\Delta_n)_{n\ge n_0}$ is nondecreasing and
unbounded, and increases by at most one at each step. Therefore, for
every integer $\Delta\ge\Delta_0$, there exists $n\ge n_0$ with
$\Delta_n=\Delta$.
Fix such an $n$ and let $G^*$ be the disjoint union of $H_n$ and
$K_{1,\Delta}$. Every connected subgraph of $G^*$ lies entirely in
$H_n$ or entirely in $K_{1,\Delta}$. The graph $H_n$ is $K_{s,t}$-free, and the star
$K_{1,\Delta}$ contains no copy of $K_{s,t}$, since $s,t\ge2$.
Thus $G^*$ is $K_{s,t}$-free, and $\Delta(G^*)=\Delta$.
Restricting a B-coloring of $G^*$ to $E(H_n)$ gives a B-coloring of
$H_n$. Consequently, \eqref{eq:B-random-lower} gives
$q_B(G^*)\ge q_B(H_n)\ge C\frac{\Delta^{\gamma_{k,t}}}{\log\Delta}$.
This proves the theorem.

\section{Further Work}\label{s:further}

Let $\calF$ be a fixed nonempty family of connected bipartite graphs,
each with at least two edges, and let $k\ge0$ be an integer. Put
$
\mathcal G_k=
\{F: F\text{ is connected and bipartite},\,
|E(F)|\ge1,\, k(F)\le k\}.
$
For $p\in\{\chi_{2,\calF},q_B\}$ and a function
$f:\mathbb N\to(0,\infty)$, define
\begin{itemize}
    \item $p(\Delta,k)=O(f(\Delta))
    \iff
    \forall F\in\mathcal G_k,\quad
    \limsup\limits_{\Delta\to\infty}
    \frac{p(\Delta,F)}{f(\Delta)}<\infty$;
    \item $p(\Delta,k)=\Omega(f(\Delta))
    \iff
    \exists F\in\mathcal G_k,\quad
    \liminf\limits_{\Delta\to\infty}
    \frac{p(\Delta,F)}{f(\Delta)}>0$.
\end{itemize}
We write $p(\Delta,k)=\Theta(f(\Delta))$ when both relations hold.
We use $\chi_{2,H}(\Delta,k)$ when $\calF=\{H\}$.

Our results give improved bounds for $\chi_{2,\calF}(\Delta,F)$ and
$q_B(\Delta,F)$ in terms of the structural parameter $k(F)$
introduced in this paper.
In the following two subsections, we discuss the two coloring problems
separately and formulate conjectures on the asymptotic bounds for
$\chi_{2,\calF}(\Delta,k)$ and $q_B(\Delta,k)$ in different ranges of $k$.

\subsection{$(2,\calF)$-avoiding colorings}

Chuet's \textbf{Problem~\ref{prob:A}} asks which bipartite exclusions
force an improvement over the general upper bound for frugal coloring.
In \textbf{Theorem~\ref{thm:main}}, we prove that
$\chi_{2,\calF}(\Delta,k)=
O((\frac{\Delta^m}{\log\Delta})^{\frac{1}{m-1}})$
for $0\le k\le m-2$, where $m\ge2$ is the minimum number of edges
in a member of $\calF$.
For frugal coloring, taking $\calF=\{K_{1,\beta+1}\}$ gives
$m=\beta+1$, and hence a positive answer to
\textbf{Problem~\ref{prob:A}} whenever $k(F)\le\beta-1$.
Moreover, \textbf{Corollary~\ref{cor:girth}} answers
\textbf{Problem~\ref{prob:C}} affirmatively, while
\textbf{Theorem~\ref{thm:lower-avoiding}} shows that the bound in
\textbf{Corollary~\ref{cor:girth}} has the correct order for fixed tree
obstructions under any
fixed prescribed girth. A natural next question is whether the threshold $m-2$ is sharp for
$\chi_{2,\calF}(\Delta,k)$. Note that the answer is affirmative for
$m=2$: in this case $\chi_{2,\calF}(G)=\chi(G^2)$, and since $k(C_4)=1$, the point--line incidence graphs of projective
planes, which are $C_4$-free, give
$\chi_{2,\calF}(\Delta,1)=\Omega(\Delta^2)$~\cite{HindMolloyReed1997}.
We believe that the answer remains affirmative for $m\ge3$ and
propose the following conjecture.

\begin{conjecture}\label{conj:avoiding-threshold}
Let $\mathcal H$ be a fixed nonempty family of connected bipartite graphs,
let $m\ge2$ be the minimum number of edges in a member of $\mathcal H$,
and let $k\ge0$ be an integer.
As $\Delta\to\infty$,
\[
\chi_{2,\mathcal H}(\Delta,k)=
\begin{cases}
O\!\left((\frac{\Delta^m}{\log\Delta})^{\frac{1}{m-1}}\right),
& \text{if }k\le m-2,\\[4pt]
\Theta\!\left(\Delta^{\frac{m}{m-1}}\right),
& \text{if }k\ge m-1.
\end{cases}
\]
\end{conjecture}

Note that an affirmative answer to \textbf{Conjecture~\ref{conj:avoiding-threshold}} would also answer
\textbf{Problem~\ref{prob:B}}.

\subsection{B-colorings}

\textbf{Theorem~\ref{thm:B-k}} gives
$q_B(\Delta,k)=O(\Delta)$ for $k\le1$ and
$q_B(\Delta,k)=O(\frac{\Delta^{2-\frac{1}{k}}}{\log\Delta})$
for $k\ge2$.
When $k\le1$, the linear order is best possible. Indeed, every B-coloring
is a proper edge-coloring, so $q_B(\Delta,k)\ge\Delta$.
When $k\ge2$, the lower bounds in
\textbf{Corollary~\ref{cor:B-k-sharp}} show that the upper bound in
\textbf{Theorem~\ref{thm:B-k}} is nearly sharp: for every $\varepsilon>0$,
$q_B(\Delta,k)=
\Omega(\frac{\Delta^{2-\frac{1}{k}-\varepsilon}}{\log\Delta})$
as $\Delta\to\infty$.
We believe that the upper bound in \textbf{Theorem~\ref{thm:B-k}}
gives the correct asymptotic order and propose the following conjecture.

\begin{conjecture}\label{conj:B-order}
Let $k\ge0$ be an integer. As $\Delta\to\infty$,
\[
q_B(\Delta,k)=
\begin{cases}
\Theta(\Delta), & \text{if }k\le1,\\[4pt]
\Theta\!\left(\frac{\Delta^{2-\frac{1}{k}}}{\log\Delta}\right),
& \text{if }k\ge2.
\end{cases}
\]
\end{conjecture}

Another direction is to determine which fixed bipartite graphs $F$
satisfy $q_B(\Delta,F)=O(\Delta)$ as $\Delta\to\infty$.
Linear bounds for $q_B(G)$ are also known under other structural
restrictions. Write $\Delta=\Delta(G)$.
Gy\'arf\'as, Martin, Ruszink\'o, and
S\'ark\"ozy~\cite{GyarfasMartinRuszinkoSarkozy2024} proved
$q_B(G)\le2\Delta+8$ for planar graphs, $q_B(G)\le2\Delta$ for
bipartite planar graphs, and $q_B(G)\le\Delta+1$ for outerplanar
graphs with $\Delta\ge4$.
For planar graphs, Kong, Wang, and Zheng~\cite{KongWangZheng2026}
proved $q_B(G)\le2\Delta+6$, and the stronger bounds
$q_B(G)\le2\Delta+4$ when $\Delta\ge12$ and
$q_B(G)\le2\Delta$ when $\Delta\ge38$.
They also proved $q_B(G)=\Delta$ for outerplanar graphs with
$\Delta\ge7$. Chen, Wang, and Xu~\cite{ChenWangXu2026} obtained
$q_B(G)\le2\Delta$ for maximal planar graphs of diameter two with
$\Delta\ge5$. Vuolo~\cite{Vuolo2026} proved that every planar graph
satisfies $q_B(G)\le\Delta+\max\{\Delta_2(G),38\}$, where
$\Delta_2(G)=\max\limits_{u,v\in V(G),\,u\ne v}|N(u)\cap N(v)|$
is the maximum codegree.

More generally, Hu, Kong, and Wang~\cite{HuKongWang2026} proved that
every $d$-degenerate graph with $1\le d\le\Delta$ satisfies
$q_B(G)\le\Delta+(d-1)\Delta_2(G)\le d\Delta$, with equality
$q_B(K_{d,\Delta})=d\Delta$. For $K_{2,t}$-free graphs, this gives
$q_B(G)\le\Delta+(d-1)\min\{t-1,\Delta\}$, which also proved by
Jiang~\cite{Jiang2026}.
For $K_{2,t}$-free planar graphs, Jiang further proved
$q_B(G)\le\Delta+t-1$ whenever $t\ge35$ or $\Delta>428$.
He also showed that $q_B(G)=\Delta$ if $t=2$ and $\Delta\ge7$,
or if $t\ge3$ and $\Delta\ge14(t-1)$.

These results motivate the following problem for bipartite
exclusions.

\begin{problem}\label{prob:linear-classification}
Which fixed bipartite graphs $F$ with at least one edge satisfy
$q_B(\Delta,F)=O(\Delta)$ as $\Delta\to\infty$?
\end{problem}

\noindent\textbf{Acknowledgement}

This research was supported in part by the National Natural Science Foundation of China (Nos. 12471330 and 12571373) and the Shandong University Youth Student Basic Research Project.

\noindent\textbf{Declaration on the Use of Generative AI}

All key mathematical ideas, problem formulations, proof strategies, arguments, and results presented in this paper were developed by the authors. Generative AI tools were used solely to assist with the exploration of some proof strategies and calculations and with auxiliary tasks, including checking typographical and linguistic errors, improving the organization and clarity of the manuscript, and drawing attention to possible inconsistencies or gaps that were subsequently examined by the authors. The authors independently verified all statements and proofs and take full responsibility for the content of this paper.

\bibliography{cas-refs}

\end{document}